\documentclass[12pt]{amsart}
\usepackage{amsfonts}
\usepackage{amsmath}
\usepackage{amssymb}
\usepackage[margin=1.2in]{geometry}
\newtheorem{theorem}{Theorem}[section]
\newtheorem{corollary}[theorem]{Corollary}
\newtheorem{lemma}[theorem]{Lemma}
\newtheorem{proposition}[theorem]{Proposition}

\theoremstyle{definition}

\newtheorem{remark}[theorem]{Remark}
\newtheorem{example}[theorem]{Example}
\newtheorem{claim}{Claim}

\numberwithin{equation}{section}

\begin{document}
\title[Complex Tauberian theorems for Laplace transforms]{Complex Tauberian theorems for Laplace transforms with local pseudofunction boundary behavior}

\author[G. Debruyne]{Gregory Debruyne}
\thanks{G. Debruyne gratefully acknowledges support by Ghent University, through a BOF Ph.D. grant}
\address{G. Debruyne\\ Department of Mathematics\\ Ghent University\\ Krijgslaan 281 Gebouw\\ B 9000 Gent\\ Belgium}
\email{gdbruyne@cage.UGent.be}
\author[J. Vindas]{Jasson Vindas}
\thanks{The work of J. Vindas was supported by the Research Foundation--Flanders, through the FWO-grant number 1520515N}
\address{J. Vindas\\ Department of Mathematics\\ Ghent University\\ Krijgslaan 281\\ B 9000 Gent\\ Belgium}
\email{jvindas@cage.UGent.be}

\begin{abstract} We provide several Tauberian theorems for Laplace transforms with local pseudofunction boundary behavior. Our results generalize and improve various known versions of the Ingham-Fatou-Riesz theorem and the Wiener-Ikehara theorem. 
Using local pseudofunction boundary behavior enables us to relax boundary requirements to a minimum. Furthermore, we allow possible null sets of boundary singularities and remove unnecessary uniformity conditions occurring in earlier works; to this end, we obtain a useful characterization of local pseudofunctions. Most of our results are proved under one-sided Tauberian hypotheses; in this context, we also establish new boundedness theorems for Laplace transforms with pseudomeasure boundary behavior. As an application, we refine various results related to the Katznelson-Tzafriri theorem for power series.  \end{abstract}

\maketitle

\section{Introduction}
Complex Tauberian theorems for Laplace transforms have been strikingly useful tools in diverse areas of mathematics such as number theory and spectral theory for differential operators \cite{korevaarbook,shubin2001}. Many developments in complex Tauberian theory  from the last three decades  have been motivated by applications in operator theory and semigroups. We refer to the monographs \cite[Chap.~4]{a-b-h-n} and \cite[Chap.~III]{korevaarbook} for complete accounts on the subject and its historical developments (see also the expository article \cite{korevaar2002}). Some recent results can be found in \cite{revesz-roton,seifert,zhang2014}; see \cite{chill} for connections with the theory of almost periodic functions.

Much work on complex Tauberians centers around two groups of statements, usually labeled as Fatou-Riesz theorems or Wiener-Ikehara theorems, and an extensively studied and central problem is that of taking boundary requirements on the Laplace transforms and/or the Tauberian hypotheses on the analyzed functions to a minimum.

The goal of this article is to considerably improve various complex Tauberian theorems for Laplace transforms and power series.  In particular, we shall refine and extend a number of results from \cite{a-o-r,a-b,ingham1935,karamata1934,katznelson,korevaar2002,korevaar2005,korevaar2005FR,s-v}. Most of the theorems from those articles can be considered as generalizations of the classical version of the Fatou-Riesz theorem for Laplace transforms by Ingham \cite{ingham1935} that we state below, or as extensions of the Katznelson-Tzafriri theorem \cite{katznelson} for power series which we will generalize in Section \ref{section power series} (Theorem \ref{psth2}). Our improvements consist, on the one hand, in relaxing the boundary behavior of Laplace transforms (power series) to local pseudofunction behavior, with possibly exceptional null sets of boundary singularities, and, on the other hand, by simultaneously considering one-sided Tauberian conditions on the functions (sequences). It should be pointed out that the use of pseudofunctions in Tauberian theory was initiated by the seminal work of Katznelson and Tzafriri \cite{katznelson}. More recently, Korevaar has written a series of papers \cite{korevaar2002,korevaar2005,korevaar2005FR} that emphasize the role of local pseudofunction boundary behavior as optimal boundary condition in complex Tauberian theorems for Laplace transforms, see also his book \cite{korevaarbook}. 

We mention that in \cite{Debruyne-VindasPNTEquivalences} we have obtained applications of our Tauberian theorems from this article to the study of PNT equivalences for Beurling's generalized numbers, generalizing results by Diamond and Zhang from \cite{diamond-zhang2012}. In that context we show that local pseudofunction boundary behavior appears as a natural condition in the analysis of properties of Beurling zeta functions.

In order to motivate and outline the content of the paper, let us state here two representative results that we shall generalize. We start with the aforementioned Tauberian theorem of Ingham from \cite{ingham1935}, which we formulate in slightly more general terms than its original form.
Let us first fix some terminology. A real-valued 
function $\tau$ is called \emph{slowly decreasing} \cite{korevaarbook} if for each $\varepsilon > 0$ there is $\delta > 0$ such that
\begin{equation}
\label{tintroeq1} \liminf_{x\to\infty}\inf_{h\in[0,\delta]}(\tau(x+h) - \tau(x)) > - \varepsilon.
\end{equation}
We extend the definition of slow decrease to complex-valued functions by requiring that their real and imaginary parts are both slowly decreasing.
An analytic function $G(s)$ on $\Re e\: s>0$ is said to have $L^{1}_{loc}$-boundary behavior on $\Re e\: s=0$ if  $\lim_{\sigma\to 0^{+}} G(\sigma +it)$ exists in $L^{1}(I)$ for any finite interval $I\subset \mathbb{R}$. This is of course the case if $G$ has analytic or continuous extension to $\Re e\: s=0$.
We also point out that Laplace transforms below are given by improper integrals.
\begin{theorem}[Ingham \cite{ingham1935}]
\label{Inghamth} Let $\tau \in L^{1}_{loc}(\mathbb{R})$ be slowly decreasing, vanish on $(-\infty,0)$, and have convergent Laplace transform 
\begin{equation}
\label{tintroeq2}
\mathcal{L}\{\tau;s\}=\int_{0}^{\infty}\tau(x)e^{-sx}\mathrm{d}x \quad \mbox{ for }\ \Re e \: s > 0.
\end{equation}
Suppose that there is a constant $b$ such that
\begin{equation*}
\mathcal{L}\{\tau;s\}- \frac{b}{s}
\end{equation*}
has $L^{1}_{loc}$-boundary behavior on $\Re e\:s=0$, then $\lim_{x\to\infty} \tau(x)=b.$
\end{theorem}

Special cases of Theorem \ref{Inghamth} were also proved by Karamata \cite{karamata1934}, and notably by Newman in connection with his attractive simple proof of the PNT via contour integration \cite{newman,newmanbook}. Newman's method was later adapted to other Tauberian problems in numerous articles, see e.g. \cite{a-o-r,a-b,korevaar1982,korevaar2005FR,korevaar2006,zagier1997} and the various bibliographical remarks in \cite[Chap.~III]{korevaarbook}. In particular, Arendt and Batty \cite{a-b} gave the following Tauberian theorem, which is a version of Theorem \ref{Inghamth} for absolutely continuous $\tau(x)=\int_{0}^{x}\rho(u)\mathrm{d}u$ with the more restrictive two-sided Tauberian hypothesis that $\rho(x)$ is bounded. Nevertheless, they allowed a (closed) null set of possible boundary singularities.
\begin{theorem}[Arendt and Batty \cite{a-b}]
\label{K-ABth} Let $\rho \in L^{\infty}(\mathbb{R})$ vanish on $(-\infty,0)$.  Suppose that $\mathcal{L}\{\rho; s\}$ has analytic continuation at every point of the complement of $i E$ where $E\subset \mathbb{R}$ is a closed null set. If $0\notin iE$ and  
\begin{equation}
\label{tintroeq3}
\sup_{t\in E} \sup_{x>0}\left| \int_{0} ^{x} e^{-i tu}\rho(u)\mathrm{d}u\right|<\infty,
\end{equation}
then the (improper) integral of $\rho$ converges to $b=\mathcal{L}\{\rho; 0\}$, that is, 
\begin{equation}
\label{tintroeq4}
\int_{0}^{\infty}\rho(x)\mathrm{d}x=b.
\end{equation}
\end{theorem}
A power series version of Theorem \ref{K-ABth} was obtained by Allan, O'Farrell, and Ransford in \cite{a-o-r}. Korevaar \cite{korevaar2005FR} also gave a version of Theorem \ref{K-ABth} employing the less restrictive local pseudofunction boundary behavior (but without allowing boundary singularities).

We shall prove the ensuing Tauberian theorem of which Theorem \ref{Inghamth} and Theorem \ref{K-ABth} are particular instances. Define 
\begin{equation}
\label{eqlogcoeff}
D_{j}(\omega)=\frac{d^{j}}{dy^{j}}\left.\left(\frac{1}{\Gamma(y)}\right)\right|_{y=\omega}.
\end{equation}
We refer to Section \ref{section preli} for the definition of local pseudofunction boundary behavior and some background material on related concepts. 
\begin{theorem}
\label{Fatou-Riesz1}  Let $\tau \in L^{1}_{loc}(\mathbb{R})$ be slowly decreasing, vanish on $(-\infty,0)$, and have convergent Laplace transform $(\ref{tintroeq2})$.
Let $\beta_1\leq \dots\leq \beta_m\in [0,1)$ and $k_1,\dots,k_{m}\in\mathbb{Z}_{+}$.
\begin{enumerate}
\item [(i)] If the analytic function 
\begin{equation*}
\mathcal{L}\{\tau;s\}- \frac{a}{s^{2}}- \sum_{n=1}^{N}\frac{b_n}{s-it_n}- \sum_{n=1}^{m}\frac{c_n+d_n\log^{k_n}\left(1/s\right)}{s^{\beta_n+1}} \quad\quad (t_{n}\in\mathbb{R})
\end{equation*}
has local pseudofunction boundary behavior on $\Re e\:s=0$, then 
\begin{equation*}
\tau(x)= ax+\sum_{n=1}^{N}b_{n}e^{it_n x}+ \sum_{n=1}^{m} x^{\beta_n}\left(\frac{c_n}{\Gamma(\beta_n+1)}+d_n\sum_{j=0}^{k_n} \binom{k_n}{j}D_{j}(\beta_n+1)\log^{k_n-j}x \right)+o(1).
\end{equation*}
\item [(ii)] Suppose that there is a closed null set $E\subset \mathbb{R}$ such that:
 \begin{enumerate}
 \item [(I)] The analytic function
\begin{equation*}
\mathcal{L}\{\tau;s\}- \sum_{n=1}^{N}\frac{b_n}{s-it_n} \quad\quad (t_{n}\in\mathbb{R})
\end{equation*}
 has local pseudofunction boundary behavior on the open subset $i(\mathbb{R}\setminus E)$ of $\Re e\:s=0$,
  \item [(II)] for every $t\in E$ there is $M_t>0$ such that
\begin{equation}
\label{tintroeq6}
\sup_{x>0}\left| \int_{0}^{x}\tau(u)e^{-itu}\mathrm{d}u\right|<M_{t},
\end{equation}
and
 \item [(III)] $E\cap \left\{t_1,\dots,t_N\right\} =\emptyset.$
\end{enumerate}
Then 
\begin{equation*}
\tau(x)= \sum_{n=1}^{N}b_{n}e^{it_n x}+o(1).
\end{equation*}

 \end{enumerate}

\end{theorem}

We actually obtain more general Laplace transform versions of the Ingham-Fatou-Riesz theorem than Theorem \ref{Fatou-Riesz1} in Section \ref{section tauberians Laplace}, where we also study one-sided Tauberian conditions other than slow decrease. In particular, we prove there a Tauberian theorem for very slowly decreasing functions \cite{korevaarbook} which only requires knowledge of the boundary behavior of the Laplace transform near the point $s=0$. We also give a finite form version of Theorem \ref{K-ABth} for bounded functions, which is applicable when information about the Laplace transform is available just on a specific boundary line segment.  Furthermore, we shall provide in Section \ref{section tauberians Laplace} a generalization of the Wiener-Ikehara theorem where boundary singularities are allowed; this result extends Korevaar's distributional version of the Wiener-Ikehara theorem from \cite{korevaar2005}.

As is well known, one-sided Tauberian conditions usually demand a more delicate treatment than two-sided ones. Our main technical tool in this respect is boundedness theorems for Laplace transforms of boundedly decreasing functions with local pseudomeasure boundary behavior in a neighborhood of $s=0$; such boundedness results are discussed in Section \ref{section boundedness theorem}.  We mention that a special case of Theorem \ref{tbth1} was stated by Korevaar in \cite[Prop.~III.10.2, p.~143]{korevaarbook}; however, his proof contains mistakes (cf. Remark \ref{tbrk1} below).

Note that unlike (\ref{tintroeq3}) we do not require any uniformity assumptions on the bounds (\ref{tintroeq6}) for $t$ in the exceptional set $E$. The elimination of the uniformity condition will be achieved with the aid of Romanovski's lemma, a simple but powerful topological lemma originally devised for removing transfinite induction arguments in the construction of the Denjoy integral \cite{Romanovski}, and that usually has very interesting applications in analysis when combined with the Baire theorem \cite{estrada-vindas2010R, estrada-vindas2012GInt,gordonbook}. 

The investigation of singular boundary sets in Tauberian theorems such as Theorem \ref{tintroeq3}(ii) has led us to a characterization of local pseudofunctions, which  we discuss in Section \ref{section pseudofunctions}. Once this characterization is established, the Tauberian theorems from Section \ref{section tauberians Laplace} are shown via simple arguments in combination with the boundedness theorems from Section \ref{section boundedness theorem}.

Section \ref{section power series} is devoted to Tauberian theorems for power series that generalize results by Katznelson and Tzafriri \cite{katznelson}, Allan, O'Farrell, and Ransford \cite{a-o-r}, and Korevaar \cite{korevaarbook}.

Finally, we mention that we state all of our results for scalar-valued functions, but in most cases one can readily verify that analogous versions are also valid for functions with values in Banach spaces if the one-sided Tauberian conditions are replaced by their two-sided counterparts; we therefore leave the formulations of such generalizations to the reader.

\section{Preliminaries}
\label{section preli}
We collect in this section several useful notions that play a role in the article.
\subsection{Distributions and Fourier transform}
We shall make extensive use of standard Schwartz distribution calculus in our manipulations. Background material on distribution theory and Fourier transforms can be found in many classical textbooks, e.g. \cite{bremermann, hormander1990, vladimirov}; see \cite{p-s-v,vladimirov-d-z1} for asymptotic analysis and Tauberian theorems for generalized functions. 

If $U\subseteq\mathbb{R}$ is open, $\mathcal{D}(U)$ is the space of all smooth functions with compact support in $U$; its topological dual $\mathcal{D}'(U)$ is the space of distributions on $U$.
The standard Schwartz test function space of rapidly decreasing functions is denoted as usual by $\mathcal{S}(\mathbb{R})$, while $\mathcal{S}'(\mathbb{R})$ stands for the space of tempered distributions. The dual pairing between a distribution $f$ and a test function $\varphi$ is denoted as $\langle f, \varphi\rangle$, or  as $\langle f(x), \varphi(x)\rangle$ with the use of a dummy variable of evaluation. Locally integrable functions are regarded as distributions via $\langle f(x),\varphi(x)\rangle=\int_{-\infty}^{\infty}f(x)\varphi(x)\mathrm{d}x$. 

We fix the constants in the Fourier transform as
$\hat{\varphi}(t)=\mathcal{F}\{\varphi;t\}=\int_{-\infty}^{\infty}e^{-itx}\varphi(x)\:\mathrm{d}x.$ Naturally, the Fourier transform is well defined on $\mathcal{S}'(\mathbb{R})$ via duality, that is, the Fourier transform of $f\in\mathcal{S}'(\mathbb{R})$ is the tempered distribution $\hat{f}$ determined by $\langle \hat{f}(t),\varphi(t)\rangle=\langle f(x),\hat{\varphi}(x)\rangle$.
If $f\in\mathcal{S}'(\mathbb{R})$ has support in $[0,\infty)$, its Laplace transform is
$\mathcal{L}\left\{f;s\right\}=\left\langle f(u),e^{-su}\right\rangle,$ analytic on $\Re e\:s>0$, and its Fourier transform $\hat{f}$ is the distributional boundary value of $\mathcal{L}\left\{f;s\right\}$ on $\Re e\:s=0$ (see Subsection \ref{subsection dvb}). See \cite{vladimirov,zemanian} for complete accounts on Laplace transforms of distributions.

\subsection{Local pseudofunctions and pseudomeasures} Pseudofunctions and pseudo\-measures arise in connection with various problems from harmonic analysis \cite{benedettobook,Katznelson2004}.
A tempered distribution $f\in\mathcal{S}'(\mathbb{R})$ is called a (global) \emph{pseudomeasure} if $\hat{f} \in L^{\infty}(\mathbb{R})$. If we additionally have $\lim_{|x|\to\infty}\hat{f}(x)=0$, we call $f$ a (global) \emph{pseudofunction}. Particular instances of pseudomeasures are any finite (complex Borel) measure on $\mathbb{R}$, while any element of $L^{1}(\mathbb{R})$ is a special case of a pseudofunction. Note that the space of all pseudomeasures $PM(\mathbb{R})$ is the dual of the Wiener algebra $A(\mathbb{R}):=\mathcal{F}(L^{1}(\mathbb{R}))$. The space of pseudofunctions is denoted as $PF(\mathbb{R})$. Notice also that $PM(\mathbb{R})$ and $PF(\mathbb{R})$ have a natural module structure over the Wiener algebra, if $f\in PM(\mathbb{R})$ ($f\in PF(\mathbb{R})$) and $g\in A(\mathbb{R})$, their multiplication $fg$ is the distribution determined in Fourier side as $2\pi \widehat{fg}= \hat{f}\ast \hat{g}$.

We say that a distribution $f$ is a \emph{local} pseudofunction at $x_0$ if the point possesses an open neighborhood where $f$ coincides with a pseudofunction, we employ the notation $f\in PF_{loc}(x_0)$. We then say that $f$ is a local pseudofunction on an open set $U$ if $f\in PF_{loc}(x_0)$ for every $x_0\in U$; we write $f\in PF_{loc}(U)$.  Likewise, one defines the spaces of local pseudomeasures $PM_{loc}(x_0)$ and $PM_{loc}(U)$ and the local Wiener algebra $A_{loc}(U)$. One can easily check that local pseudofunctions are characterized by a generalized Riemann-Lebesgue lemma \cite{korevaar2005}. A distribution
 $f$ is a local pseudofunction on $U$ if and only if $e^{iht}f(t)=o(1)$ as $\left|h\right|\to\infty$ in $\mathcal{D}'(U)$, that is, for each $\varphi\in\mathcal{D}(U)$
\begin{equation}
\label{eqRL}
\left\langle f(t),e^{iht}\varphi(t)\right\rangle=o(1), \quad |h|\to\infty .
\end{equation}
Indeed, $f\in PF_{loc}(U)$ if and only if $\varphi f\in PF(\mathbb{R})$, for each $\varphi\in\mathcal{D}(U)$, which is a restatement of (\ref{eqRL}). If we replace $o(1)$ by $O(1)$ in (\ref{eqRL}), we obtain a characterization of local pseudomeasures.  Naturally, every $f\in L^{1}_{loc}(U)$ is an example of a local pseudofunction; in particular, every continuous, smooth, or real analytic function is a local pseudofunction. Furthermore, any Radon measure on $U$ is an instance of a local pseudomeasure.  

Let us also point out that the elements of $A_{loc}(U)$ are multipliers for $PF_{loc}(U)$ and $PM_{loc}(U)$. Since every smooth function belongs locally to the Wiener algebra, the $C^{\infty}$-functions are multipliers for local pseudofunctions and pseudomeasures. 

\subsection{Boundary values of analytic functions}
\label{subsection dvb}
Let $F(s)$ be analytic on the half-plane $\Re e\:s>\alpha$. We say that $F$ has distributional boundary values  on the open set $\alpha+iU$ of the boundary line $\Re e\:s=\alpha$ if $F(\sigma+it)$ tends to a distribution $f\in\mathcal{D}'(U)$ as $\sigma\to\alpha^{+}$, that is, if 
\begin{equation*}
\lim_{\sigma\to\alpha^{+}}\int_{-\infty}^{\infty}F(\sigma+it)\varphi(t)\mathrm{d}t=\left\langle f(t),\varphi(t)\right\rangle\ , \quad \mbox{for each } \varphi\in\mathcal{D}(U).
\end{equation*}

Analytic functions admitting distributional boundary values can be characterized in a very precise fashion via bounds near the boundary. One can show \cite[pp. 63--66]{hormander1990} that $F(s)$ has distributional boundary values on $\alpha+iU$ if for a fixed $\sigma_{0}>\alpha$ and for each bounded open $U'\subset U$  there are $N=N_{U'}$ and $M=M_{U'}$ such that 
$$
|F(\sigma+it)|\leq \frac{M}{(\sigma-\alpha)^{N}}\:, \quad \sigma+it\in (\alpha,\sigma_{0}]+iU',
$$
which is a result that goes back to the work of K\"{o}the. We refer to the textbooks \cite{bremermann,PilipovicK,C-M} for further details on boundary values and generalized functions; see also the article \cite{d-p-v} for recent results. 

Finally, we say that $F$ has \emph{local} pseudofunction (local pseudomeasure) boundary behavior on $\alpha+iU$ if it has distributional boundary values on this boundary set  
and the boundary distribution $f\in PF_{loc}(U)$ ($f\in PM_{loc}(U)$). The meaning of having pseudofunction (pseudomeasure) boundary behavior at the boundary point $\alpha+it_0$ is $f\in PF_{loc}(t_0)$ ($f\in PM_{loc}(t_0)$), i.e., $F$ has such local boundary behavior on a open line boundary segment containing  $\alpha+it_0$. We emphasize again that $L^1_{loc}$-boundary behavior, continuous, or analytic extension are very special cases of local pseudofunction and pseudomeasure boundary behavior.

\section{Boundedness theorems}
\label{section boundedness theorem}
We prove in this section \emph{boundedness} Tauberian theorems for Laplace transforms involving local pseudomeasure boundary behavior. 

Our first result is a very important one, as the rest of the article is mostly built upon it. It extends early boundedness theorems by Karamata \cite[Satz II]{karamata1934} and Korevaar \cite[Prop.~III.10.2, p.~143]{korevaarbook}, which were obtained under continuous or $L^{1}_{loc}$-boundary behavior, respectively. Here we take the local boundary requirement of the Laplace transform to a minimum\footnote{Clearly (\ref{teq2}) implies that $\mathcal{L}\{\tau;s\}$ has local pseudomeasure boundary behavior on $\Re e\:s=0$.} by relaxing it to local pseudomeasure boundary behavior at $s=0$.

The next notion plays a key role as Tauberian condition for boundedness. We say that a 
real-valued 
function $\tau$ is \emph{boundedly decreasing} \cite{binghambook,revesz-roton} (with additive arguments) if
there is a $\delta>0$ such that
$$
\liminf_{x\to\infty} \inf_{h\in[0,\delta]}(\tau(x+h)-\tau(x))>-\infty,
$$
that is, if there are constants $\delta,x_0,M > 0$ such that 
\begin{equation}
\label{boundedly decreasing}
 \tau(x+h) - \tau(x) \geq -M, \quad \text{for } 0 \leq h \leq \delta \text{ and } x \geq x_0.
\end{equation}
Bounded decrease for a complex-valued function means that its real and imaginary parts are boundedly decreasing.

\begin{theorem} \label{tbth1} Let $\tau \in L^{1}_{loc}(\mathbb{R})$ vanish on $(-\infty,0)$ and have convergent Laplace transform
\begin{equation}
\label{eqL1}
\mathcal{L}\{\tau;s\}=\int_{0}^{\infty}\tau(x)e^{-sx}\mathrm{d}x \quad \mbox{for }\ \Re e \: s > 0.
\end{equation}
Suppose that one of the following two Tauberian conditions is satisfied:
\begin{itemize}
 \item [$(T.1)$] $\tau$ is boundedly decreasing.
 \item [$(T.2)$] There are $x_0\geq 0$ and $\beta\in\mathbb{R}$ such that $e^{\beta x}\tau(x)$ is  non-negative and non-decreasing on $[x_0,\infty)$.
\end{itemize}
If $\mathcal{L}\{\tau;s\}$ has pseudomeasure boundary behavior at $s = 0$, then
\begin{equation}
\label{teq2}
\tau(x) = O(1), \quad x \to \infty.
\end{equation}
\end{theorem}

\begin{proof} We show the theorem under the Tauberian hypotheses $(T.1)$ and $(T.2)$ separately. Set $F(s):=\mathcal{L}\{\tau;s\}$. Let $i(-\lambda,\lambda)$ be an open line segment sufficiently close to $s=0$ where the local pseudomeasure boundary behavior of $F$ is fulfilled. We may assume that $\tau$ is real-valued, because both $2\mathcal{L}\{\Re e\:\tau;s\}= F(s)+\overline{F(\overline{s})}$ and $2i\mathcal{L}\{\Im m\:\tau;s\}= F(s)-\overline{F(\overline{s})}$ have local pseudomeasure boundary behavior on $i(-\lambda,\lambda)$.

\smallskip

\textit{The Tauberian condition $(T.1)$.} Note that, by iterating the inequality (\ref{boundedly decreasing}) and enlarging the constant $M$ if necessary, we may suppose that 
\begin{equation}  
\label{eqcond1bis}
\tau(x+h) - \tau(x) > -M(h+1)
\end{equation} 
for $x > x_0$ and $h>0$. Since modifying $\tau$ on a bounded interval does not affect the local pseudomeasure behavior (indeed, the Laplace transform of a compactly supported function is entire), we may actually assume that (\ref{eqcond1bis}) holds for all $x,h > 0$. By adding a positive constant to $\tau$, we may also assume that $\tau(0)\geq 0$.  We divide the rest of the proof into four main steps.

\medskip

\emph{Step 1.} The first step in the proof is to translate the local pseudomeasure boundary behavior hypothesis into a convolution average condition for $\tau$. We show that  
\begin{equation} \label{tpeq1}
\int_{-\infty}^{\infty} \tau(x+h)\psi(x)dx = O(1),
\end{equation}
for all non-negative $\psi\in \mathcal{F}(\mathcal{D}((-\lambda,\lambda))$.

 For it, set
\begin{equation}
\label{eqAuxFunction1}
g(x) :=  \tau(x)+ M(x+1) \quad  \mbox{for }\: x > 0,\:  \mbox{ and }\: 0 \: \mbox{ elsewhere}.
\end{equation}
In view of (\ref{eqcond1bis}) and $\tau(0)\geq 0$, we have that $g$ is a positive function. Clearly $\tau(x)e^{-\sigma x}\in \mathcal{S}'(\mathbb{R}_x)$, for each $\sigma>0$. Let $\psi\in\mathcal{S}(\mathbb{R})$ be a non-negative test function  whose Fourier transform has support in $(-\lambda,\lambda)$. 
 By the monotone convergence theorem, the relation $\mathcal{L}\{\tau; \sigma + it\} = \mathcal{F}\{\tau e^{-\sigma \cdot};t\}$, which holds in $\mathcal{S}'(\mathbb{R})$, and the fact that $F(s)$ has distributional boundary values in $i(-\lambda,\lambda)$,  we obtain
\begin{align*}
\int^{\infty}_{-\infty} g(x+h) \psi(x) \mathrm{d}x& = \lim_{\sigma \rightarrow 0^{+}} \int^{\infty}_{0} g(x) \psi(x-h) e^{-\sigma x} \mathrm{d}x \\
& = \lim_{\sigma \rightarrow 0^{+}} \frac{1}{2\pi}\left\langle \mathcal{L}\{\tau; \sigma + it\}, e^{iht}\hat{\psi}(-t)\right\rangle +M\int^{\infty}_{-h} (x+1+h) \psi(x) \mathrm{d}x
\\
&
=  \frac{1}{2\pi}\left\langle F(it),e^{iht}\hat{\psi}(-t)\right\rangle+ M\int^{\infty}_{-h} (x+1+h) \psi(x) \mathrm{d}x.
\end{align*}
Subtracting the very last term from both sides of the above equality and using the fact that $\left\langle F(it),e^{iht}\hat{\psi}(-t)\right\rangle=O(1)$, which follows from the local pseudomeasure boundary behavior of $F$, we have proved that (\ref{tpeq1}) holds 
for all non-negative $\psi\in \mathcal{F}(\mathcal{D}((-\lambda,\lambda))$. 

From now on, we fix in the convolution average estimate (\ref{tpeq1}) a non-negative even function $\psi \in \mathcal{F}(\mathcal{D}(-\lambda,\lambda))$ with 
$$\int_{-\infty}^{\infty} \psi(x)\mathrm{d}x = 1.$$

 \medskip
 
\emph{Step 2.} The second step consists in establishing the auxiliary estimate
\begin{equation}\label{eqAuxTau1}\tau(x)=O(x).
\end{equation}
To show this bound, we employ again the auxiliary function $g$ defined in $(\ref{eqAuxFunction1})$. We have that $g$ is positive and satisfies the rough average bound 
$$
\displaystyle\int_{-\infty}^{\infty} g(x+h)\psi(x)dx = O(h),
$$
due to (\ref{tpeq1}). Notice $g(x+h)-g(h)\geq -M$, it then follows that
\begin{align*}
 0\leq g(h) & = 2 \int^{\infty}_{0} g(h) \psi(x) \mathrm{d}x \leq 2 \int^{\infty}_{0} (g(x+h)+M) \psi(x) \mathrm{d}x 
\\ &  \leq M+ 2 \int^{\infty}_{-\infty} g(x+h) \psi(x) \mathrm{d}x= O(h),
\end{align*}
and thus also $\tau(h) = O(h)$. 

\medskip

\emph{Step 3.} In this crucial step we prove that  $\tau$ is bounded from above by contradiction. Suppose then that $\tau$ is not bounded from above. Let $X>2$ be so large that $\int^{X}_{-X} \psi(x)\mathrm{d}x \geq 3/4$ and $\int_{X}^{\infty}x^{2}\psi(x)\mathrm{d}x< 1$. Choose $C\geq1$ witnessing the $O$-constant in (\ref{tpeq1}), namely, 
\begin{equation}
\label{eq1proofbdd}
\left|\int_{-\infty}^{\infty} \tau(x+h) \psi(x)\mathrm{d}x\right| \leq C, \quad \forall h\geq 0.
\end{equation}
Let $R$ be arbitrarily large; in fact, we assume that
\begin{equation}
\label{eqRchoice}
R> 4C+4M\left(X+1+\int_{X}^{\infty}x\psi(x)\mathrm{d}x\right). 
\end{equation}
A key point to generate a contradiction is to show that unboundedness from above of $\tau$ forces the existence of a large value $y$ satisfying the maximality assumptions from the ensuing claim:

\begin{claim}
\label{tclaim1} If $\tau$ is unbounded from above, there is $y$ such that $\tau(y) \geq R$, $\tau(x) < 2\tau(y)$ when $x \leq y$ and $\tau(x) \leq \tau(y)(x+X-y)^{2}$ whenever $x \geq y$.
\end{claim}
Indeed, by the assumption that $\tau$ is not bounded from above, we may choose $y_{0}$ such that $\tau(y_{0}) \geq R$. Suppose that $y_{0}$ does not satisfy the requirements of the claim. This means the following set is non-empty,
\begin{equation*}
V_{0} :=  \{x \mid \tau(x) \geq 2\tau(y_{0}) \text{ and }  x \leq y_{0}\} \cup  \{x \mid \tau(x) \geq \tau(y_{0})(x+X-y_{0})^{2} \text{ and } x \geq y_{0} \}.
\end{equation*}
Since $\tau(x) = O(x)$, we have that $V_{0}$ is contained in some bounded interval. Let us choose $y_{1} \in V_{0}$. If $y_1$ does not satisfy the properties of the claim, we may define $V_{1}$ in a similar fashion. Iterating the procedure, we either find our $y$ or can construct recursively a sequence of points $y_{n+1}\in V_n$, where the sets
$$
V_{n} :=  \{x \mid \tau(x) \geq 2\tau(y_{n}) \text{ and }  x \leq y_{n}\} \cup  \{x \mid \tau(x) \geq \tau(y_{n})(x+X-y_{n})^{2} \text{ and } x \geq y_{n} \}
$$ 
are non-void.
 We will show that this procedure breaks down after finitely many steps, i.e., some $V_{n}$ must be empty, which would show Claim \ref{tclaim1}. It suffices to prove that $V_{1} \subseteq V_{0}$. In fact, it would then follow that $\dots\subseteq V_{n} \subseteq V_{n-1} \subseteq \dots \subseteq V_{0}$, thus all $V_{n}$ would live in the same bounded interval. But on the other hand if no $V_{n}$ would be empty we would obtain that $\tau (x_n)\geq 2^{n}R$; consequently $\tau$ would be  unbounded on this bounded interval, which contradicts (\ref{eqAuxTau1}). It remains thus to show $V_{1} \subseteq V_{0}$. If $y_{1} \leq y_{0}$, this is very easy to check. If $y_{1} \geq y_{0}$, the verification for $x \leq y_{1}$ is still easy. We thus assume that  $x \geq y_{1}\geq y_0$. We have to prove that $\tau(x) \geq \tau(y_{0})(x+X-y_{0})^{2}$ provided that $x \in V_{1}$. We have
\begin{align*}
 \tau(x) &\geq \tau(y_{1})(x+X-y_{1})^{2} 
 \geq \tau(y_{0})(y_{1}+X-y_{0})^{2}(x+X-y_{1})^{2}
  \\&
  \geq \tau(y_{0})(x+2X-y_{0})^{2} \geq \tau(y_{0})(x+X-y_{0})^{2},
\end{align*}
where we have used the inequality $a^{2}b^{2} \geq (a+b)^{2}$ which certainly holds for $a,b\geq 2$. This concludes the proof of the claim. 

\smallskip

 We now use (\ref{tpeq1}) and Claim \ref{tclaim1} to produce the desired contradiction and to conclude that $\tau$ is bounded from above. Let $y$ be as in Claim \ref{tclaim1}. 
We set $h = X + y$ in (\ref{eq1proofbdd}) and we are going to split the integral $\int_{-\infty}^{\infty} \tau(x+X+y) \psi(x) \mathrm{d}x$ in two parts. By the choice of $R$ (cf. (\ref{eqRchoice})) and (\ref{eqcond1bis}) (with $h=x+X$ and $y$ instead of $x$), the contribution on the interval $[-X,\infty]$ is larger than
\begin{align*}
\int_{-X}^{\infty} \tau(x+X+y) \psi(x) \mathrm{d}x&\geq \frac{3\tau(y)}{4}- M\int_{-X}^{\infty} (x+X+1) \psi(x) \mathrm{d}x
\\
&
\geq \frac{3\tau(y)}{4}-\frac{R}{4}+C
\\
&
\geq
 \frac{\tau(y)}{2}+C.
\end{align*}
Combining this inequality with the upper bound from (\ref{eq1proofbdd}), we obtain
$$
\tau(y)\leq -2\int^{-X}_{-\infty}\tau(x+X+y)\psi(x)\mathrm{d}x\leq 2\sup_{t\in[0,y]}(-\tau(t))\int_{X}^{\infty}\psi(x)\mathrm{d}x\leq \frac{1}{4}\sup_{t\in[0,y]}(-\tau(t)).
$$
In particular, we conclude that there exists $t<y$ which is ``very negative'' with respect to $-\tau(y)$, that is, $\tau(t)\leq -3\tau(y)$.

 Applying a similar argument with $h=t-X$, we derive
$$
\int_{-\infty}^{X}\tau(x+t-X)\psi(x)\mathrm{d}x\leq \frac{R}{4}-C-\frac{9}{4}\tau(y)\leq -C-2\tau(y),
$$
which, together with the lower bound in (\ref{eq1proofbdd}) for $h=t-X$, yields
$$
\tau(y)\leq \frac{1}{2}\int_{X}^{\infty}\tau(x+t-X) \psi(x)\mathrm{d}x\leq \frac{1}{2} \sup_{u\in[t,\infty)} \frac{\tau(u)}{(u-t+X)^{2}} \int_{X}^{\infty}x^{2}\psi(x)\mathrm{d}x.
$$

We have therefore found a ``very positive'' value $\tau(u)$ for $u > t$, i.e., one where $\tau$ satisfies $\tau(u) \geq 2 \tau(y) (u-t +X)^{2}$. This $u$ contradicts the maximality assumptions on $y$ from Claim \ref{tclaim1} (in both cases $u \leq y$ and $u \geq y$). So $\tau$ is bounded from above. 
 
\medskip

\emph{Step 4.} Finally, we establish the lower bound with the aid of the upper one. Find $C'$ such that $\tau(x)\leq C'$ for all $x$. Using that $\psi$ is even and non-negative and the lower bound in (\ref{eq1proofbdd}), we then have
\begin{align*}
-C&\leq \int_{-h}^{\infty} \tau(x+h)\psi(x)\mathrm{d}x\leq \frac{C'}{2}+\int_{-h}^{0} \tau(x+h)\psi(x)\mathrm{d}x
\\
&
= \frac{C'}{2}+\frac{\tau(h)}{2}+ \int_{0}^{h}(\tau(h-x)-\tau(h))\psi(x)\mathrm{d}x
\\
&
\leq \frac{C'}{2}+\frac{\tau(h)}{2}+ M\int_{0}^{\infty}(x+1)\psi(x)\mathrm{d}x,
\end{align*}
which yields the lower bound. This concludes the proof of the theorem under $(T.1)$.

\bigskip

\emph{The Tauberian condition $(T.2)$}. The proof under the Tauberian condition $(T.2)$ is much simpler. We may assume that $\beta>0$; otherwise, $\tau$ is non-decreasing and in particular boundedly decreasing. Using the positivity of $\tau$, one can establish as above (\ref{tpeq1}) for all non-negative $\psi\in\mathcal{S}(\mathbb{R})$ with $\operatorname*{supp} \hat{\psi} \subset (-\lambda,\lambda)$. As before, we choose $\psi$ with $\int_{-\infty}^{\infty}\psi(x)\mathrm{d}x=1$. Set $C=\int_{0}^{\infty}\psi(x)e^{-\beta x}\mathrm{d}x>0$. Since $\tau(h)\leq e^{\beta x}\tau(x+h)$ for $x\geq 0$, we obtain

$$
\tau(h)=\frac{1}{C}\int_{0}^{\infty}\tau(h)e^{-\beta x} \psi(x)\mathrm{d}x\leq \frac{1}{C}\int_{-\infty}^{\infty} \tau(x+h)\psi(x)\mathrm{d}x=O(1).
$$

\end{proof}
\begin{remark}\label{tbrk1} Korevaar states in \cite[Prop.~III.10.2, p.~143]{korevaarbook} a weaker version of Theorem under $(T.1)$ for Laplace transforms with $L^{1}_{loc}$-boundary behavior on the whole line $\Re e\:s=0$; however, his proof turns out to have a major gap. In fact, Korevaar's argument is based on the analysis of $\beta_{x}:=\sup_{t>0}e^{-xt}|\tau(t)|$, $x>0$. He further reasons by contradiction and states for his analysis that he may assume that $\beta_{x}=\sup_{t>0}e^{-xt}\tau(t)$; however, the case $\beta_{x}=\sup_{t>0}-e^{-xt}\tau(t)$ cannot be treated analogously, being actually the most technically troublesome one (compare with our proof above and Karamata's method from \cite{karamata1934}). 
\end{remark}
\begin{remark}\label{tbrk2} The point $s=0$ plays an essential role in Theorem \ref{tbth1}, in the sense that, in general, pseudomeasure boundary behavior of the Laplace transform in a neighborhood of any other point $it_{0}\neq0$ of $\Re e\: s=0$ does not guarantee boundedness of $\tau$. A simple example is provided by $\tau(x)=x$, $x>0$, whose Laplace transform $1/s^{2}$ has local pseudomeasure boundary behavior on $i(\mathbb{R}\setminus\left\{0\right\})$.
\end{remark}

The Tauberian condition
\begin{equation}
\label{teq2ingham}
\limsup_{x\to\infty}e^{-\theta x}\left|\int_{0^{-}}^{x}e^{\theta u}\mathrm{d}\tau(u)\right|<\infty \quad\quad (\theta>0),
\end{equation}
where $\tau$ is assumed to be of local bounded variation,
appeared in Ingham's work \cite[Thm.~I]{ingham1935} in connection to his Fatou-Riesz type theorem for Laplace transforms.
\begin{corollary}
\label{tbc1}  Let $\tau$ vanish on $(-\infty,0)$, be of local bounded variation on $[0,\infty)$, and have convergent Laplace transform $(\ref{eqL1})$ admitting 
pseudomeasure boundary behavior at the point $s = 0$. Suppose that there is $\theta>0$ such that
$$
T_{\theta}(x):= e^{-\theta x}\int_{0^{-}}^{x}e^{\theta u}\mathrm{d}\tau(u)\quad \mbox{is bounded from below.}
$$
Then,
\begin{equation}
\label{tbeq10}
\tau(x) = T_{\theta}(x)+ O(1), \quad x \to \infty.
\end{equation}
In particular, $\tau$ is bounded if $(\ref{teq2ingham})$ holds.

\end{corollary}
\begin{proof}  Noticing that
\begin{equation}
 \label{tbeq12}
 \tau(x)=T_{\theta}(x)+\theta\int_{0}^{x}T_{\theta}(u)\mathrm{d}u,
 \end{equation}
it is enough to show that $T^{(-1)}_{\theta}(x)=\int_{0}^{x}T_{\theta}(u)\mathrm{d}u$ is bounded, which would yield (\ref{tbeq10}). We have that 
\begin{equation}
\label{tbeq13}
\mathcal{L}\{T^{(-1)}_{\theta};s\}=\frac{\mathcal{L}\{T_{\theta};s\}}{s}=\frac{\mathcal{L}\{\tau;s\}}{s+\theta}. 
\end{equation}
The function $1/(s+\theta)$ is $C^{\infty}$ on $\Re e\: s=0$, and thus a multiplier for local pseudomeasures. Therefore, $\mathcal{L}\{T^{(-1)}_{\theta};s\}$ has pseudomeasure boundary behavior at $s = 0$. Since $T^{(-1)}_{\theta}$ is boundedly decreasing, we obtain $T^{(-1)}_{\theta}(x)=O(1)$ from Theorem \ref{tbth1}.
\end{proof}

 Theorem \ref{tbth1} can be further generalized if we notice that $(T.1)$ is invariant under addition and subtraction of \emph{boundedly oscillating} functions. We call a function $\tau$ boundedly oscillating if there is $\delta>0$ such that
\begin{equation}
\label{boeq1}
\limsup_{x\to\infty} \sup_{h\in[0,\delta]}|\tau(x+h)-\tau(x)|<\infty.
\end{equation}
For example, Ingham's condition (\ref{teq2ingham}) is a particular case of bounded oscillation (cf. (\ref{tbeq12})). Moreover, noticing that the property (\ref{boeq1}) is equivalent to $f=\exp\circ \tau \circ \log $ being $O$-regularly varying \cite[p.~65]{binghambook}, we obtain from the Karamata type representation theorem for the latter function class \cite[p.~74]{binghambook} that any (measurable) boundedly oscillating function $\tau$ can be written as
\begin{equation}
\label{repbo}
\tau(x)=\int_{0}^{x}g(y)\mathrm{d}y +O(1), \quad g\in L^{\infty}[0,\infty),
\end{equation}
for  $x\in[x_0,\infty)$, for some large enough $x_0$. Although we shall not use the following fact in the future, we point out that one can actually choose $g$ in (\ref{repbo}) enjoying much better properties:
\begin{proposition}
\label{proprbo} If $\tau$ is boundedly oscillating and measurable, then $(\ref{repbo})$ holds for some $g\in C^{\infty}(\mathbb{R})$ vanishing on $(-\infty,0]$ and satisfying $g^{(n)}\in L^{\infty}(\mathbb{R})$ for all $n\in\mathbb{N}$.
\begin{proof}
Bounded oscillation implies that $|\tau(x+h)-\tau(x)|<M (h+1)$ for some $M>0$, all $h\geq0$, and all sufficiently large $x$. We may assume that this inequality holds for all $x$ and that $\tau$ vanishes, say, on $(-\infty,1]$. Take a non-negative $\varphi\in\mathcal{D}(0,1)$ with $\int_{0}^{1}\varphi(x)\mathrm{d}x=1$. The $C^{\infty}$-function $f(x)=\int_{-\infty}^{\infty}\tau(x+y)\varphi(y)\mathrm{d}y$ has support in $(0,\infty)$, $f(x)=\tau(x)+O(1)$ and $f^{(n)}(x)\in L^{\infty}(\mathbb{R})$ for $n\geq1$. Thus $g=f'$ satisfies all requirements.
\end{proof}
\end{proposition}

We have the ensuing extension of  Theorem \ref{tbth1}.

\begin{theorem}
\label{tbth2} Let $\tau \in L^{1}_{loc}(\mathbb{R})$ vanish on $(-\infty,0)$, have convergent Laplace transform $(\ref{eqL1})$, and be boundedly decreasing. Then, the function $\tau$ is boundedly oscillating if and only if there is $G(s)$ analytic on the intersection of $\Re e\:s>0$ with a (complex) neighborhood of $s=0$ such that
\begin{equation}\label{eq3extra}
\mathcal{L}\{\tau;s\}-\frac{G(s)}{s} \quad\mbox{and} \quad G(s)
\end{equation}
both admit pseudomeasure boundary behavior at $s=0$. 

Furthermore, $\tau$ has the asymptotic behavior $(\ref{repbo})$ with $g\in L^{\infty}(\mathbb{R})$ given in terms of the Fourier transform by the distribution
\begin{equation}
\label{tbeq14}
\hat{g}(t)=\lim_{\sigma\to 0^{+}}G(\sigma+it) \quad \mbox{ in }\mathcal{D}'(-\lambda,\lambda),
\end{equation}
for sufficiently small $\lambda>0$.
\end{theorem}

\begin{proof} If $\tau$ already satisfies (\ref{repbo}) (i.e., it is boundedly oscillating), $\mathcal{L}\{\tau;s\}-G(s)/s$ and $G(s)$ clearly have local pseudomeasure behavior, where $G$ is the Laplace transform of $g$. Conversely, by applying the edge-of-the-wedge theorem  \cite{bremermann,rudin1971} and the fact that the analytic function $1/s$ has global pseudomeasure boundary behavior, we may assume that the $L^{\infty}$-function determined by (\ref{tbeq14}) has support on $[0,\infty)$ and $G(s)=\mathcal{L}\{g;s\}$. The function $\tau(x)-\int_{0}^{x}g(y)\mathrm{d}y$ is of bounded decrease, Theorem \ref{tbth1} then yields (\ref{repbo}).
\end{proof}

The next result is a special case of Theorem \ref{tbth2}; nevertheless, it has a very useful form for applications. It is a version of our Fatou-Riesz Theorem \ref{Fatou-Riesz1}(i) where the asymptotic estimate is obtained with an $O(1)$-remainder. 

\begin{theorem}
\label{tbth3} Let $\tau \in L^{1}_{loc}(\mathbb{R})$ be boundedly decreasing, vanish on $(-\infty,0)$, and have convergent Laplace transform
$(\ref{eqL1})$. Suppose that 
\begin{equation}
\label{tbeq15}
\mathcal{L}\{\tau;s\}- \frac{a}{s^{2}}-\sum_{n=0}^{N}\frac{b_n+c_n\log^{k_n}\left(1/s\right)}{s^{\beta_n+1}}
\end{equation}
has pseudomeasure boundary behavior at $s=0$, where the $\beta_n<1$ and the $k_n\in\mathbb{Z}_{+}$. Then,
\begin{equation}
\label{tbeq16}
\tau(x)= ax+ \sum_{n=0}^{N} x^{\beta_n}\left(\frac{b_n}{\Gamma(\beta_n+1)}+c_n\sum_{j=0}^{k_n} \binom{k_n}{j}D_{j}(\beta_n+1)\log^{k_n-j}x \right)+O(1),
\end{equation}
$x\to\infty,$ where $D_j(\omega)$ is given by $(\ref{eqlogcoeff})$.
\end{theorem}
Naturally, only those $\beta_n\geq 0$  deliver a contribution to (\ref{tbeq16}).
\begin{proof}
Terms with $\beta_n<0$ or $b_n/s$ in (\ref{tbeq15}) are pseudomeasures. The result is a direct consequence of Theorem \ref{tbth2} (or Theorem \ref{tbth1}) after noticing that the Laplace transform of $x^{\mu}_{+}$ is $s^{-\mu-1}\Gamma(\mu+1)$ and that of 
$$
x^{\mu}\sum_{j=0}^{m} \binom{m}{j}D_{j}(\mu+1)\log^{m-j}_{+}x
$$
 is $s^{-\mu-1}\log^{m}(1/s)$ plus an entire function (see e.g. \cite[Lemma~5]{delange1954}). The first function is boundedly oscillating if $\mu\leq 1$, while the second one if $\mu<1$ for all positive integers $m$.
\end{proof}

It is important to point out that Theorem \ref{tbth2} and Theorem \ref{tbth3} are no longer true if one replaces bounded decrease by the Tauberian hypothesis $(T.2)$ from Theorem \ref{tbth1}, as shown by the following simple example.

\begin{example}
\label{tbex1}
Consider the non-negative function 
$$
\tau(x)=x\left(1+\frac{\cos x}{2}\right).
$$
Since $\tau(x)+\tau'(x)\geq 0$, we have that $e^{x}\tau(x)$ is non-decreasing. Its Laplace transform satisfies
$$
\mathcal{L}\{\tau;s\}-\frac{1}{{s}^{2}}=\frac{1}{4(s-i)^{2}}+\frac{1}{4(s+i)^{2}}
$$
and therefore has analytic continuation through $i(-1,1)$; in particular, it has  local pseudomeasure boundary behavior on this line segment. However,
 $$
 \tau(x)=x+\Omega_{\pm}(x), \quad x\to\infty.
 $$
\end{example}

\section{A characterization of local pseudofunctions}\label{section pseudofunctions}
We now turn our attention to a characterization of distributions that are local pseudo\-functions on an open set $U\subseteq\mathbb{R}$. Let $f\in\mathcal{D}'(U)$. Its singular pseudofunction support in $U$, denoted as $\operatorname*{sing\: supp}_{PF} f$, is defined as the complement in $U$ of the largest open subset of $U$ where $f$ is a local pseudofunction; a standard argument involving partitions of the unity and the fact that smooth functions are multipliers for local pseudofunctions show that this notion is well defined. The ensuing theorem is the main result of this section.
\begin{theorem}
\label{cpfth1} Let $f\in\mathcal{D}'(U)$. Suppose there is a closed null set $E\subset U$ such that
\begin{enumerate}
\item [(I)] $ \operatorname*{sing\: supp}_{PF} f\subseteq E$, and
\item [(II)] for each $t_{0}\in E$ there is a neighborhood $V_{t_{0}}$ of $t_{0}$ and a local pseudomeasure $f_{t_{0}}\in PM_{loc}(V_{t_0})$ such that 
\begin{equation}
\label{cpfeq1}
f=(t-t_{0})f_{t_0} \quad \mbox{on } V_{t_0}\:.  
\end{equation}
\end{enumerate}
Then,  $\operatorname*{sing\: supp}_{PF}f=\emptyset$, that is, $f$ is a local pseudofunction on $U$.
\end{theorem}

Naturally, the converse of Theorem \ref{cpfth1} is trivially true, as one can take for $E$ the empty set.

Before giving a proof of Theorem \ref{cpfth1}, we discuss a characterization of distributions that `vanish' at $\pm\infty$ in the sense of Schwartz \cite{schwartz} (or have S-limit equal to 0 at $\pm\infty$  in the terminology of S-asymptotics from \cite{p-s-v}); this result becomes particularly useful when combined with Theorem \ref{cpfth1}. Given $g\in\mathcal{S}'(\mathbb{R})$, we define its pseudofunction spectrum as the closed set $\operatorname{sp}_{PF}(g)=\operatorname*{sing\: supp}_{PF}\hat{g}$. The Schwartz space of bounded distributions $\mathcal{B}'(\mathbb{R})$ is the dual of the test function space 
$$\mathcal{D}_{L^1}(\mathbb{R})=\{\varphi\in C^{\infty}(\mathbb{R})|\: \varphi^{(n)}\in L^{1}(\mathbb{R}), \: \forall n\in\mathbb{N}\}.$$
Traditionally \cite[p.~200]{schwartz}, the completion of $\mathcal{D}(\mathbb{R})$ in (the strong topology of) $\mathcal{B}'(\mathbb{R})$ is denoted as $\dot{\mathcal{B}}'(\mathbb{R})$. A distribution $\tau$ is said to vanish at $\pm\infty$ if $\tau\in\dot{\mathcal{B}}'(\mathbb{R})$; the latter membership relation is equivalent \cite[p.~512]{d-p-v2015} (cf. \cite[p.~201--202]{schwartz}) to the convolution average condition 
\begin{equation}
\label{cpfeq2}
\lim_{|h|\to\infty} \langle \tau(x+h),\varphi(x)\rangle= \lim_{|h|\to\infty} (\tau\ast \check{\varphi})(h)=0, \end{equation}
for each test function $\varphi\in\mathcal{S}(\mathbb{R})$. We also refer to \cite{d-p-v2015} for convolution average characterizations of wider classes of distribution spaces in terms of translation-invariant Banach spaces of tempered distributions. We then have,

\begin{proposition}
\label{cpfp1} Let $\tau\in \mathcal{B}'(\mathbb{R})$. Then, 
 $\tau\in \dot{\mathcal{B}}'(\mathbb{R})$ if and only if $\operatorname{sp}_{PF}(\tau)=\emptyset.$
\end{proposition}
\begin{proof}
If $\tau\in \dot{\mathcal{B}}'(\mathbb{R})$, then we directly obtain  $\hat{\tau}\in PF_{loc}(\mathbb{R})$ in view of (\ref{cpfeq2}) and (\ref{eqRL}). Conversely, if $\hat\tau$ is a local pseudofunction on $\mathbb{R}$, we obtain that (\ref{cpfeq2}) holds for every $\varphi\in\mathcal{F}(\mathcal{D}(\mathbb{R}))$. On the other hand, the hypothesis $\tau\in \mathcal{B}'(\mathbb{R})$ gives that the set of translates of $\tau$ is bounded in $\mathcal{S}'(\mathbb{R})$, and hence equicontinuous by the Banach-Steinhaus theorem. The density of $\mathcal{F}(\mathcal{D}(\mathbb{R}))$ in $\mathcal{S}(\mathbb{R})$ then implies that (\ref{cpfeq2}) remains valid for all $\varphi\in\mathcal{S}(\mathbb{R})$ (in fact for all $\varphi\in\mathcal{D}_{L^1}(\mathbb{R})$), namely, $\tau\in\dot{\mathcal{B}}'(\mathbb{R})$ by the quoted characterization of the space of distributions vanishing at $\pm\infty$.
\end{proof}

The next corollary can be regarded as a Tauberian theorem for Fourier transforms. (The Tauberian condition being the membership relation $\tau\in \mathcal{B}'(\mathbb{R})$.)

\begin{corollary}
\label{cpfc1} Suppose that $\tau\in \mathcal{B}'(\mathbb{R})\cap L^{1}_{loc}(\mathbb{R})$ and that there is a closed null set $E$ such that $\operatorname{sp}_{PF}(\tau)\subseteq E$ and for each $t\in E$ one can find a constant $M_t>0$, independent of $x$, with
\begin{equation}
\label{cpfeq3}
\left|\int_{0}^{x}\tau(u)e^{-itu}\mathrm{d}u\right|\leq M_t, \quad x\in\mathbb{R}.
\end{equation}
Then, $\tau\in \dot{\mathcal{B}}'(\mathbb{R})$.
\end{corollary}

\begin{proof}
This follows from Propostion \ref{cpfp1} because Theorem \ref{cpfth1} applied to $f=\hat{\tau}$ yields $\operatorname{sp}_{PF}(\tau)=\emptyset$. Indeed, the condition (\ref{cpfeq3}) implies 
(\ref{cpfeq1}) with $f_{t_{0}}$ given by the Fourier transform of the $L^{\infty}$-function $e^{it_{0}x}\int_{0}^{x}\tau(u)e^{-it_0u}\mathrm{d}u$.
\end{proof}

The rest of this section is devoted to the proof of Theorem \ref{cpfth1}. We shall use the following variant of Romanovski's lemma.
\begin{lemma}{\cite[Thm.~2.1]{estrada-vindas2010R}}
\label{RomanovskiLemma}Let $X$ be a topological space. Let $\mathfrak{U}$ be a non-empty family of open sets of $X$ that satisfies the following four properties:

\textnormal{(a)} $\mathfrak{U}\neq\{\varnothing\}.$

\textnormal{(b)} If $V\in\mathfrak{U},$ $W\subset V,$ and $W$ is open,
then $W\in\mathfrak{U}.$

\textnormal{(c)} If $V_{\alpha}\in\mathfrak{U}$ $\forall\alpha\in A,$
then $\bigcup_{\alpha\in A}V_{\alpha}\in\mathfrak{U}.$

\textnormal{(d)} Whenever $V\in\mathfrak{U},$ $V\neq X,$ then there
exists $W\in\mathfrak{U}$ such that $W\cap\left(  X\setminus V\right)
\neq\varnothing.$

Then $\mathfrak{U}$ must be the class of all open subsets of $X.$
\end{lemma}

We also need the ensuing two lemmas.

\begin{lemma}
\label{cpfl1} Let $g\in PM(\mathbb{R})$ have compact support and let $t_0\notin\operatorname*{supp} g$. Then $(t-t_{0})^{-1}g\in PM(\mathbb{R})$.
\end{lemma}
\begin{proof}
Let $\varphi\in\mathcal{D}(\mathbb{R})$ be equal to 1 in a neighborhood of $\operatorname*{supp} g$ with $t_0\notin\operatorname*{supp} \varphi$. Then, $\psi(t)=(t-t_{0})^{-1}\varphi(t)$ is also an element of $\mathcal{D}(\mathbb{R})$ and $(t-t_{0})^{-1}g= \psi g \in PM(\mathbb{R})$.
\end{proof}

\begin{lemma}\label{cpfl2}
Let $f=\hat{\tau}$ with $\tau\in L^{\infty}(\mathbb{R})$ and let $W$ be open. Suppose that the restriction of $f$ to $W\setminus \bigcup_{j=1}^{n}[t_{j}-\ell_j/2,t_{j}+\ell_{j}/2]$ is a local pseudofunction, where $[t_{j}-\ell_{j},t_{j}+\ell _{j}]\subset W$ and the $[t_{j}-\ell_{j},t_{j}+\ell _{j}]$ are disjoint. There is an absolute constant $C$ such that 
$$
\limsup_{|h|\to\infty} \left|\langle f(t),\varphi(t)e^{iht} \rangle\right|\leq CM \| \hat{\varphi} \|_{L^{1}(\mathbb{R})}\sum_{j=1}^{n}\ell_{j}, \quad\forall \varphi\in\mathcal{D}(W),
$$
where 
$$
M=\max_{j=1,\dots,n}\sup_{x\in\mathbb{R}}\left|\int_{0}^{x}\tau(u)e^{-it_{j}u}\mathrm{d}u\right|.
$$
\end{lemma}
\begin{proof} We may obviously assume that $M<\infty$. Let $\chi\in\mathcal{D}(-1,1)$ be even such that $\chi(t)=1$ for $t$ in a neighborhood of $[-1/2,1/2]$. Set $\chi_{j}(t)=\chi((t-t_{j})/\ell_{j})$ and $\check{\tau}(x)=\tau(-x)$. Since 
$$
\operatorname*{supp}(\varphi(1-\sum_{j=1}^{n}\chi_{j}))\cap \bigcup_{j=1}^{n}[t_{j}-\ell_j/2,t_{j}+\ell_{j}/2]=\emptyset,
$$
we have that 
\begin{align*} 
\limsup_{|h|\to\infty} \left|\langle f(t),\varphi(t)e^{iht} \rangle\right|&\leq
\frac{\|\hat{\varphi}\|_{L^{1}(\mathbb{R})}}{2\pi}\sum_{j=1}^{n}\|\check{\tau}\ast \hat{\chi}_{j}\|_{L^{\infty}(\mathbb{R})} 
\\
&
=\frac{\|\hat{\varphi}\|_{L^{1}(\mathbb{R})}}{2\pi}\sum_{j=1}^{n}\ell_{j}\sup_{h\in\mathbb{R}}\left|\int_{-\infty}^{\infty} \tau(x)e^{-it_{j}x}\hat{\chi}(\ell_{j}(h+x))\mathrm{d}x\right|
\\
&
\leq 
\frac{\|\hat{\chi}'\|_{L^{1}(\mathbb{R})}}{2\pi}M\|\hat{\varphi}\|_{L^{1}(\mathbb{R})}\sum_{j=1}^{n}\ell_{j},
\end{align*}   
where we have used integration by parts in the last step.
\end{proof}

We can now show Theorem \ref{cpfth1}.

\begin{proof}[Proof of Theorem \ref{cpfth1}] 
We will apply Lemma \ref{cpfl1} to reduce the proof of the  general case to showing a special case of Corollary \ref{cpfc1}. In fact, our assumptions imply that $f$ is a local pseudomeasure on $U$. Since the hypotheses and the conclusion are local, we can assume that $f$ is the restriction to $U$ of a global compactly supported pseudomeasure $\hat{\tau}$, with $\tau\in L^{\infty}(\mathbb{R}).$ We may thus assume that $f$ is globally defined on $\mathbb{R}$ with compact support and we simply write $f=\hat{\tau}$. We can also suppose that each $f_{t_0}$ appearing in (\ref{cpfeq1}) is a compactly supported global pseudomeasure. By Lemma \ref{cpfl1} applied to $g_{t_0}:=f-(t-t_0)f_{t_0}$, we can replace $f_{t_0}$ by $(t-t_{0})^{-1}g_{t_{0}}+f_{t_0}$ and also suppose that the all equations (\ref{cpfeq1}) hold on $\mathbb{R}$ with $f_{t_{0}}\in PM(\mathbb{R})$. Since any two different pseudomeasure solutions of (\ref{cpfeq1}) can only differ by a multiple of the Dirac delta $\delta(t-t_0)$, we conclude under these circumstances that $\tau$ must fulfill (\ref{cpfeq3}) for each $t\in E$. Moreover, by going to localizations again if necessary, we assume that $E$ is compact in $U$. After all these reductions, we now proceed to show that $\operatorname*{sing\: supp}_{PF} f=\emptyset$. 

We are going to check that $f$ is a local pseudofunction on $U$ via Lemma \ref{RomanovskiLemma}. For it, consider $X=U$ and the family $\mathfrak{U}$ of all open subsets $V\subseteq U$ such that $f_{|V}\in PF_{loc}(V)$. The condition (a) holds for $\mathfrak{U}$ because of the assumption (I), while (b) and (c) are obvious. It remains to check the condition (d). So, let $V\in\mathfrak{U}$ with $V\subsetneq U$. Set $E_1=E\cap(U\setminus V)$. If $E_1=\emptyset$, we would be done because then we could find an open $W\subset U$ disjoint from the compact $E$ with $(U\setminus V)\subset W$; we would then obtain that $W\in\mathfrak{U}$ since $E$ contains $\operatorname*{sing\: supp}_{PF} f$. So, assume that the null compact set $E_1\subset E$ is non-empty. Consider the sequence of continuous functions 
$$g_{N}(t)=\max_{-N\leq x\leq N} \left| \int_{0}^{x}\tau(u)e^{-it u}\mathrm{d}u \right|, \quad t\in E_1.$$
The $g_{N}$ are pointwise bounded on $E_1$ because of (\ref{cpfeq3}). Employing the Baire theorem, we now obtain the existence of a constant $M>0$ and an open subset $W\subset U$ such that  $E_2=W\cap E_1\neq\emptyset$ and 
$$
\sup_{t\in E_2}\sup_{x\in\mathbb{R}}\left| \int_{0}^{x}\tau(u)e^{-it u}\mathrm{d}u \right|<M<\infty.
$$
By reducing the size of $W$ if necessary, we may additionally assume that $E_2$ is compact. We now show that $f_{|W}$ is a local pseudofunction. Let $\varphi\in\mathcal{D}(W)$ and fix $\varepsilon>0$. By compactness of the null set $E_2$, we can clearly find a finite covering $E_2\subseteq \bigcup_{j=1}^{n} [t_{j}-l_j/2,t_{j}+l_{j}/2]$ by intervals satisfying the conditions of Lemma \ref{cpfl2} with $\sum_{j=1}^{n}\ell_{j}<\varepsilon$. This gives that
$$
 \limsup_{|h|\to\infty} \left|\langle f(t),\varphi(t)e^{iht} \rangle\right|\leq \varepsilon CM \| \hat{\varphi} \|_{L^{1}(\mathbb{R})},
$$
namely,  $\lim_{|h|\to\infty}\langle f(t),\varphi(t)e^{iht} \rangle=0$ because $\varepsilon$ was arbitrarily chosen.
Consequently, $W$ satisfies $W\in\mathfrak{U}$ and $W\cap(U\setminus V)$ is non-empty. We have therefore shown that $\mathfrak{U}$ is the family of all open subsets of $U$; in particular,  $U\in\mathfrak{U}$, or equivalently, $\operatorname*{sing\: supp}_{PF} f=\emptyset$.
\end{proof}

\section{Tauberian theorems for Laplace transforms}
\label{section tauberians Laplace}
We now apply our previous results from Section \ref{section pseudofunctions} and Section \ref{section boundedness theorem} to derive several complex Tauberian theorems for Laplace transforms with local pseudofunction boundary behavior.

Our first main result is a general version of Theorem \ref{Fatou-Riesz1}.

\begin{theorem} \label{tth1} Let $\tau \in L^{1}_{loc}(\mathbb{R})$ with $\operatorname*{supp} \tau \subseteq [0,\infty)$ be slowly decreasing and have Laplace transform
\begin{equation}
\label{ttheq1}
\mathcal{L}\{\tau;s\}=\int_{0}^{\infty}\tau(x)e^{-sx}\mathrm{d}x \quad  \mbox{convergent for } \Re e \: s > 0.
\end{equation}
Let $g\in L^{\infty}(\mathbb{R})$ and set $G(s)=\int_{0}^{\infty}g(x)e^{-sx}\mathrm{d}x$ for $\Re e\:s>0$. Suppose that 
\begin{equation} 
 \label{ttheq3}
F(s)=\mathcal{L}\{\tau;s\}-\frac{b}{s}-\frac{G(s)}{s}
\end{equation}
has local pseudofunction boundary behavior on $i(\mathbb{R}\setminus{E})$,
where $E$ is a closed null set. If for each $it\in E$
\begin{equation}
\label{ttheqbc1}
\frac{F(s)}{s-it} \quad \mbox{has pseudomeasure boundary behavior at } it,
\end{equation}
then
\begin{equation} 
 \label{ttheq2}
\tau(x) = b+\int_{0}^{x}g(y)\mathrm{d}y +o(1), \quad x \to \infty.
\end{equation}

Conversely, if $\tau$ satisfies $(\ref{ttheq2})$, then $(\ref{ttheq3})$ has local pseudofunction boundary behavior on the whole line $\Re e\:s=0$.
\end{theorem}
\begin{proof} The function $\tau(x)-b-\int_{0}^{x}g(y)$ is also slowly decreasing, we may therefore assume that $g=0$ and $b=0$. The hypotheses imply that $\mathcal{L}\{\tau;s\}$ has pseudomeasure boundary behavior at $s=0$, and, hence, $\tau$ should be bounded near $\infty$ in view of Theorem \ref{tbth1}. In particular, $\tau\in \mathcal{B}'(\mathbb{R})$, as the sum of a compactly supported distribution and an $L^{\infty}$-function. Its Laplace transform then has distributional boundary value $\hat{\tau}$ on the whole $i\mathbb{R}$. Theorem \ref{cpfth1} hence yields $\operatorname*{sp}_{PF} (\tau)=\emptyset$, and Proposition \ref{cpfp1} gives $\tau\in\dot{\mathcal{B}}'(\mathbb{R})$, namely,
 $$
\int_{-\infty}^{\infty}\tau(x+h) \phi(x)\mathrm{d}x = o(1), \quad h\to\infty, \quad \mbox{for each } \phi \in \mathcal{S}(\mathbb{R}).
$$
It remains to choose suitable test functions in the above relation to get $\tau(x) = o(1)$. Let $\varepsilon > 0$ be arbitrary. Because $\tau$ is slowly decreasing, there exists $\delta > 0$ such that $\tau(u) - \tau(y) > - \varepsilon$ for all $ \delta+y>u>y$ and sufficiently large $y$.  Let us choose a non-negative $\phi \in \mathcal{D}(-\delta,0)$ such that $\int_{-\infty}^{\infty} \phi(x)\mathrm{d}x = 1$. Then,

\begin{equation*}
 \liminf_{h \to \infty} \tau(h) 
  =  \liminf_{h \to \infty}  \int_{-\delta}^{0} ( \tau(h) - \tau(x+h)) \phi(x) \mathrm{d}x + \int_{-\delta}^{0} \tau(x+h) \phi(x) \mathrm{d}x  \geq -\varepsilon,
\end{equation*}
Since $\varepsilon$ was arbitrary, we get $\liminf_{h\to\infty} \tau(h) \geq 0$. By a similar reasoning (now with a test function having  $\operatorname*{supp}\phi\subset(0,\delta)$), we obtain that $\limsup_{h\to\infty} \tau(h) \leq 0$, which shows that $\tau(x)=o(1)$, $x\to\infty$. 

The converse is trivial, because $F$ must then be the sum of a global pseudofunction and the Fourier transform of a compactly supported distribution (and the latter is an entire function).
\end{proof}

Note that Theorem \ref{Fatou-Riesz1} directly follows from Theorem \ref{tth1}. In fact, for Theorem \ref{Fatou-Riesz1}(i) one can argue exactly as in proof of Theorem \ref{tbth3}. For Theorem \ref{Fatou-Riesz1}(ii),  one easily sees that (II) and (III) imply (\ref{ttheqbc1}) at every $it\in iE$ with $G(s)=\sum_{n=1}^{N}it_nb_{n}(s-it_n)^{-1}$ and $b=\sum_{n=1}^{N}b_n$. More generally, if the function $g$ in Theorem \ref{tth1} has a bounded primitive, then a sufficient condition for (\ref{ttheqbc1}) is that for every $t\in E$ one can find $M_{t}>0$ with
\begin{equation}
\label{tthbceq1}
\left|\int_{0}^{x}\tau(u)e^{-itu}\mathrm{d}u\right|\leq M_{t} \quad \mbox{ and }\quad \left|\int_{0}^{x}g(u)e^{-itu}\mathrm{d}u\right|\leq M_{t}, \quad x\in\mathbb{R}.
\end{equation}

Theorem \ref{tth1} actually provides a characterization of those slowly decreasing functions that belong to an interesting subclass of the slowly oscillating functions. Given $\tau$ and $\delta>0$, define the non-decreasing subadditive function
\begin{equation}
\label{eqsg}
\Psi(\delta):=\Psi(\tau,\delta)= \limsup_{x\to\infty}\sup_{h\in (0,\delta]}|\tau(x+h)-\tau(x)|.
\end{equation}
Note that a function is boundedly oscillating precisely when $\Psi$ is finite for some $\delta$, while it is slowly oscillating if $\Psi(0^{+})=\lim_{\delta\to0^{+}}\Psi(\delta)=0$. We shall call a function \emph{R-slowly oscillating} (regularly slowly oscillating) if $\limsup_{\delta\to0^{+}}\Psi(\delta)/\delta<\infty$. Since $\Psi$ is subadditive, it is easy to see the latter implies that $\Psi$ is right differentiable at $\delta=0$ and indeed $\Psi'(0^{+})=\sup_{\delta>0}\Psi(\delta)/\delta$. It turns out that a measurable function $\tau$ is R-slowly oscillating if and only if it admits the representation (\ref{ttheq2}). This fact is known (apply the representation theorem for E-regularly varying functions \cite[Thm.~2.2.6, p.~74]{binghambook} to $\exp\circ \tau \circ \log $), but we take a small detour to give a short proof with the aid of functional analysis: 

\begin{proposition}\label{proprrso}  If $\tau$ is R-slowly oscillating and measurable, then $\tau$ can be written as $(\ref{ttheq2})$ in a neighborhood of $\infty$ where $g\in L^{\infty}(\mathbb{R})\cap C(\mathbb{R})$ and $b$ is a constant.
\begin{proof}That the function (\ref{eqsg}) is globally $O(\delta)$ implies the existence of a sequence $\{x_n\}_{n=1}^{\infty}$ tending to $\infty$ such that $|\tau(x+h)-\tau(x)|\leq C/n$ for all $0<h\leq 1/n$ and $x\geq x_n$, where $C>\Psi'(0^{+})$ is a fixed constant. Modifying $\tau$ on a finite interval if necessary, we may assume that $x_{1}=1$ and that $\tau$ vanishes on $(-\infty,1]$.
Take a non-negative $\varphi\in\mathcal{D}(0,1)$ with $\int_{0}^{1}\varphi(x)\mathrm{d}x=1$ and define the sequence of $C^{\infty}$-functions $f_{n}(x):=\int_{-\infty}^{\infty}\tau(x+y)n\varphi(ny)\mathrm{d}y= \int_{0}^{1}\tau(x+y/n)\varphi(y)\mathrm{d}y$. They have support in $[0,\infty)$ and satisfy 
\begin{equation}\label{anothereq}
|f_m(x)-\tau(x)|\leq C/n \quad \mbox{for all } m\geq n \mbox{ and } x\geq x_n.
\end{equation}
Furthermore, 
\[
|f'_{n}(x)|= n\left| \int_{0}^{1}(\tau(x+y/n)-\tau(x))\varphi'(y)\mathrm{d}y\right|\leq C\|\varphi'\|_{L^{1}}=C'.
\]
 Applying the Banach-Alaoglu theorem to $\{f'_{n}\}_{n=1}^{\infty}$ (regarded as a sequence in the bidual of $C_{b}(\mathbb{R})$) and the Bolzano-Weierstrass theorem to $\{f_{n}(0)\}_{n=1}^{\infty}$, there are subsequences such that $f_{n_k}(0)\to b$ and $f'_{n_k}\to g$ weakly in the space of continuous and bounded functions. In particular $f'_{n_k}\to g$ pointwise, $|g(x)|\leq C'$ for all $x$, and $f_{n_k}(x)$ converges uniformly to $\int_{0}^{x}g(y)dy+b$ for $x$ on compacts. We obtain from the last convergence and (\ref{anothereq}) that $|\tau(x)-b-\int_{0}^{x}g(y)dy|\leq C/n_k$ for $x\geq x_{n_k}$. 
\end{proof}
\end{proposition} 

Summarizing, part of Theorem \ref{tth1} might be rephrased as follows: A (measurable) slowly decreasing function $\tau$ is R-slowly oscillating if and only if it has convergent Laplace transform on $\Re e\:s>0$ such that (\ref{ttheq3}) has local pseudofunction boundary behavior on $\Re e\:s=0$ for some constant $b$ and some $G$ with global pseudomeasure boundary behavior. 

We now obtain an intermediate Tauberian theorem between Theorem \ref{tbth2} and Theorem \ref{tth1}, where the requirement on the Laplace transform in Theorem \ref{tth1} is relaxed to pseudofunction boundary behavior at $s=0$, but the Tauberian condition is strengthened to very slow decrease \cite{korevaarbook}. 
A real-valued 
function $\tau$ is said to be \emph{very slowly decreasing} if there is $\delta>0$ such that
\begin{equation}
\label{ttheq5}
\liminf_{x \to \infty} \inf_{h \in [0,\delta]} \tau(h+x) - \tau(x) \geq 0.
\end{equation}
As usual, the notion makes sense for complex-valued functions if we require both real and imaginary parts to be very slowly decreasing. 
Our result also involves very slow oscillation.  
 A function is called \emph{very slowly oscillating} if both $\tau$ and $-\tau$ are very slowly decreasing; or equivalently if the function (\ref{eqsg}) vanishes at some $\delta$. (This actually implies that $\Psi(\delta)=0$ for all $\delta>0$, due to subadditivity). For a measurable function $\tau$, being very slowly oscillating is equivalent to $\exp\circ \tau \circ \log $ being a Karamata slowly varying function, i.e., to the apparently weaker property
$$
\tau(x+h)=\tau(x)+o_{h}(1), \quad x\to\infty,
$$  
for each $h>0$, as follows from the well known uniform convergence theorem \cite[p.~6]{binghambook}. It also follows \cite[p.~12]{binghambook} that any (measurable) function $\tau$ is very slowly oscillating if and only if it  admits the representation
\begin{equation}
\label{repvso}
\tau(x)=b+\int_{0}^{x}g(y)\mathrm{d}y +o(1), \quad \mbox{with } \lim_{y\to\infty}g(y)=0
\end{equation}
and  a constant $b$, for $x$ in a neighborhood of $\infty$. Naturally, one can also apply the same proof method from Proposition \ref{proprbo} to show that the function $g$ in (\ref{repvso}) may be chosen to be additionally $C^{\infty}$ with all derivatives tending to $0$ at $\infty$. 

After these preparatory remarks, we are ready to state the second main Tauberian theorem from this section. 

\begin{theorem} \label{tth2} Let $\tau \in L^{1}_{loc}(\mathbb{R})$ vanish on $(-\infty,0)$, have convergent Laplace transform $(\ref{ttheq1})$, and be such that $\tau$ is very slowly decreasing. Then, the function $\tau$ is very slowly oscillating if and only if there are a constant $b'$ and $G(s)$ analytic on the intersection of $\Re e\:s>0$ with a (complex) neighborhood of $s=0$ such that  
\begin{equation}
\label{ttheq7}
\mathcal{L}\left\{\tau;s\right\}-\frac{b'}{s}-\frac{G(s)}{s} \quad\mbox{and} \quad G(s)
\end{equation}
both admit pseudofunction boundary behavior at $s=0$. 

Moreover, $\tau$ has the asymptotic behavior $(\ref{repvso})$ with $g$ given in terms of the Fourier transform by the distribution
\begin{equation}
\label{ttheq8}
\hat{g}(t)=\lim_{\sigma\to 0^{+}}G(\sigma+it) \quad \mbox{ in }\mathcal{D}'(-\lambda,\lambda),
\end{equation}
for sufficiently small $\lambda$, and the constant 
\begin{equation}
\label{ttheqcst}
b=b'+\lim_{\sigma\to0^{+}}\int_{0}^{\infty}g(-x)e^{-\sigma x}\mathrm{d}x=b'+\lim_{\sigma\to0^{+}}\left(G(\sigma)-\int_{0}^{\infty}g(x)e^{-\sigma x}\mathrm{d}x\right).
\end{equation}
 \end{theorem}
\begin{proof} The asymptotic estimate (\ref{repvso}) easily yields local pseudofunction boundary behavior of (\ref{ttheq7}) with $b=b'$ and $G(s)=\int_{0}^{\infty}g(x)e^{-s x}\mathrm{d}x$. Conversely, applying again the edge-of-the-wedge theorem, we obtain that $G(s)-\int_{0}^{\infty}g(x)e^{-sx}\mathrm{d}x$, $\Re e\:s>0$, and $\int_{0}^{\infty}g(-x)e^{s x}\mathrm{d}x$, $\Re e\:s<0$, are analytic continuations of each other through $i(-\lambda,\lambda)$, with $\lambda$ sufficiently small, which gives in particular the existence of $b$. We can thus suppose that $g$ given by (\ref{ttheq8}) has support in $[0,\infty)$, $g(x)=o(1)$, that $G$ is its Laplace transform, and that $b'=b$. Applying Theorem \ref{tbth2}, we obtain that $\tau$ satisfies (\ref{repbo}). Replacing $\tau$ by the very slowly decreasing and bounded function $\tau(x)-b-\int_{0}^{x}g(y)\mathrm{d}y$, we may assume that $b=0$ and $G=0$. So, since $\tau(x)=O(1)$, $\tau$ is actually a tempered distribution and our hypothesis on the Laplace transform becomes $\hat{\tau}$ coincides with a pseudofunction on $(-\lambda,\lambda)$. Thus, we obtain that 
\begin{equation} \label{ttheq9}
 \int_{-\infty}^{\infty} \tau(x+h) \psi(x)\mathrm{d}x = o(1), \quad h\to\infty,
\end{equation}
for any $\psi \in \mathcal{F}(\mathcal{D}(-\lambda,\lambda))$. We may assume that $\tau$ is globally bounded, say $|\tau(x)|\leq M$, for all $x>0$.
We choose $\psi$ in (\ref{ttheq9}) to be a non-negative and even test function with $\int_{-\infty}^{\infty}\psi(x)\mathrm{d}x=1$. Fix a large $X$ ensuring $\int^{\infty}_{X} \psi(x)\mathrm{d}x < 1/4$. Let $\varepsilon>0$, the very slow decrease of $\tau$  (cf. (\ref{ttheq5})) ensures that 
\begin{equation}
\label{ttheq10}
\tau(y)-\tau(u)\geq -\varepsilon(y-u+1), \quad \mbox{for } y\geq u\geq N,
\end{equation}
for some $N$. We keep 
 $t>N+2X$. Set $h(t) = t + X$ if $\tau(t) > 0$ and $h(t) = t- X$ if $\tau(t) < 0$. Using (\ref{ttheq10}), we deduce the inequality
\begin{align*}
 \left|\int^{\infty}_{-\infty} \tau(x+h(t)) \psi(x) \mathrm{d}x\right| & \geq -\left(\int^{-X}_{-\infty} + \int^{\infty}_{X}\right) M \psi(x) \mathrm{d}x + \left|\int^{X}_{-X} \tau(x+h(t)) \psi(x) \mathrm{d}x\right| \\
& \geq -2M\int_{X}^{\infty}\psi(x)\mathrm{d}x +\frac{|\tau(t)|}{2}-(2X+1)\varepsilon,
\end{align*}
which, in view of (\ref{ttheq9}), yields $\limsup_{t\to\infty}|\tau(t)|\leq 2(\varepsilon (2X+1)+2M\int_{X}^{\infty}\psi(x)\mathrm{d}x)$. Since $\varepsilon$ was arbitrary,
$$
\limsup_{t\to\infty}|\tau(t)|\leq 4M\int_{X}^{\infty}\psi(x)\mathrm{d}x.
$$
We can now take $X\to\infty$ and obtain $\lim_{t\to\infty}\tau(t)=0$.\end{proof}

Note that Theorem \ref{tth2} applies to the case when

\begin{equation*}
\mathcal{L}\{\tau;s\}- \sum_{n=1}^{m}\frac{c_n+d_n\log^{k_n}\left(1/s\right)}{s^{\beta_n+1}}
\end{equation*}
has pseudofunction boundary behavior at $ s=0$, provided that $\beta_1\leq \dots\leq \beta_m\in [0,1)$ and $k_1,\dots,k_{m}\in\mathbb{Z}_{+}$ and $\tau$ is very slowly decreasing. In this case the conclusion reads
\begin{equation*}
\tau(x)= \sum_{n=1}^{N} x^{\beta_n}\left(\frac{c_n}{\Gamma(\beta_n+1)}+d_n\sum_{j=0}^{k_n} \binom{k_n}{j}D_{j}(\beta_n+1)\log^{k_n-j}x \right)+o(1).
\end{equation*}

The next result generalizes Korevaar's distributional version of the Wiener-Ikehara theorem \cite{korevaar2005}.

\begin{theorem}\label{Wiener-Ikehara} Let $S$ be a non-decreasing function on $[0,\infty)$ with $S(x) = 0$ for $x < 0$ such that
\begin{equation*}
 \mathcal{L}\{\mathrm{d}S;s\}= \int^{\infty}_{0^{-}} e^{-sx} \mathrm{d}S(x) \text{ converges for } \Re e \: s > \alpha>0.
\end{equation*}
Suppose that there are a closed null set $E$, constants $r_0,r_1,\dots,r_N\in\mathbb{R}$, $\theta_1,\dots,\theta_N\in\mathbb{R}$, and $t_1,\dots,t_{N}>0$ such that:
\begin{enumerate}
\item [(I)]  The analytic function 
\begin{equation}
\label{wieq1}
\mathcal{L}\{\mathrm{d}S;s\} - \frac{r_0}{s-\alpha} - \sum_{n=1}^{N}r_{n}\left(\frac{e^{i\theta_n}}{s-\alpha-it_n}+\frac{e^{-i\theta_n}}{s-\alpha+it_n}\right)
\end{equation}
admits local pseudofunction boundary behavior on the open subset $\alpha+i(\mathbb{R}\setminus E)$ of the line $\Re e\:s=\alpha$,  
\item [(II)] $E\cap \left\{0,t_1,\dots,t_N\right\} =\emptyset,$ and
 \item [(III)] for every $t\in E$ there is $M_t>0$ such that
\begin{equation}
\label{wieq2}
\sup_{x>0}\left| \int_{0}^{x}e^{-\alpha u-itu}\mathrm{d}S(u)\right|<M_{t}.
\end{equation}
\end{enumerate}
Then 
\begin{equation}
\label{wieq3}
  S(x)= e^{\alpha x}\left(\frac{r_0}{\alpha}+2\sum_{n=1}^{N}\frac{ r_n\cos (t_n x+\theta_n-\arctan(t_n/\alpha))}{\sqrt{\alpha^{2}+t_{n}^{2}}} +o(1)\right), \quad x\to\infty.
\end{equation}

Conversely, if $S$ has asymptotic behavior $(\ref{wieq3})$, then $(\ref{wieq1})$ has local pseudofunction boundary behavior on the whole line $\Re e\: s=\alpha$. 
\end{theorem}
\begin{remark}\label{tthr1} The conditions (II) and (III) above can be replaced by the weaker assumption that $F(s-\alpha)$, with $F(s)$ given by (\ref{wieq1}), satisfies (\ref{ttheqbc1}) for each $t\in E$. 
\end{remark}
\begin{proof} We may assume that $-t_1,-t_2,\dots,-t_N\notin E$ because (\ref{wieq1}) has also local pseudofunction boundary behavior at $\alpha-it_n$ due to the fact that $S$ is real-valued.  Set $\tau(x)=e^{-\alpha x}S(x)$, this function $\tau$ fulfills $(T.2)$ from Theorem \ref{tbth1}. Write $\theta'_n=\arctan(t_n/\alpha)$. For $\Re e\:s>0$,
\begin{align*}
&\mathcal{L}\{\tau;s\}-\frac{r_0}{\alpha s}-\sum_{n=1}^{N}\frac{r_n}{\sqrt{\alpha^{2}+t_n^2}}\left(\frac{e^{i(\theta_n-\theta'_n)}}{s-it_n}+\frac{e^{-i(\theta_n-\theta'_n)}}{s+it_n}\right)
\\
&
=\frac{1}{s+\alpha}\left(\mathcal{L}\{\mathrm{d}S;s+\alpha\}-\frac{r_0}{s}-\sum_{n=1}^{N}\frac{r_{n}e^{i\theta_n}}{s-it_n}+\frac{r_{n}e^{-i\theta_n}}{s+it_n}\right)
\\
&
\quad\quad -\frac{1}{s+\alpha}\left(\frac{r_0}{\alpha}+\sum_{n=1}^{N}2 r_n\Re e\left(\frac{e^{i\theta_n}}{\alpha+it_n}\right)\right),
\end{align*}
which has local pseudofunction boundary behavior on $i(\mathbb{R}\setminus E)$  because $1/(\alpha+it)$ is smooth and $C^{\infty}$-functions are multipliers for local pseudofunctions. Since $1/s$ has global pseudomeasure boundary behavior, we conclude from Theorem \ref{tbth1} that $\tau(x)=O(1)$. 
It follows that $\tau(x+h)-\tau(x)\geq -\tau(x)(1-e^{-\alpha h})\gg -h$, and thus $\tau$ is slowly decreasing. Note also that, by (\ref{wieq2}),
$$
\left|\int_{0}^{x}\tau(u)e^{-itu}\mathrm{d}u\right|=\frac{1}{\sqrt{\alpha^{2}+t^{2}}}\left|-e^{-itx}\tau(x)+\int_{0}^{x}e^{-(\alpha+it)u}\mathrm{d}S(u)\right|= O_{t}(1),
$$
for each $t\in E$.
Thus, Theorem \ref{Fatou-Riesz1} (or Theorem \ref{tth1}) implies that 
$$\lim_{x\to\infty} \tau(x)-\frac{r_0}{\alpha}-2\sum_{n=1}^{N}\frac{ r_n\cos (t_n x+\theta_n-\arctan(t_n/\alpha))}{\sqrt{\alpha^{2}+t_{n}^{2}}}=0,
$$ which completes the proof.
\end{proof}

As indicated at the introduction, Theorem \ref{K-ABth} is contained in Theorem \ref{Fatou-Riesz1}. The next corollary gives a more general version that applies to Laplace transforms of functions that are bounded from below.

\begin{corollary}
\label{improvedK-ABth}
Let $\rho\in L^{1}_{loc}(\mathbb{R})$ be bounded from below, vanish on on $(-\infty,0)$, and have convergent Laplace transform 
 for $\Re e\:s>0$.
Suppose that there is closed null set $0\notin E\subset\mathbb{R}$ such that  
\begin{equation}
\label{ttheq17}
\sup_{x>0}\left| \int_{0}^{x}\rho(u)e^{-itu}\mathrm{d}u\right|<M_{t}<\infty, 
\end{equation}
for each  $t\in E$. If there is a constant $\hat{\rho}(0)\in\mathbb{C}$ such that
\begin{equation}
\label{ttheq18} \frac{\mathcal{L}\{\rho; s\}-\hat{\rho}(0)}{s}
\end{equation}
has local pseudofunction boundary behavior on $i(\mathbb{R}\setminus E)$, then the (improper) integral $\int_{0}^{\infty}\rho(u)\mathrm{d}u$ converges and 
 \begin{equation}
\label{ttheq19} \int_{0}^{\infty}\rho(u)\mathrm{d}u=\hat{\rho}(0).
\end{equation}
\end{corollary}
\begin{remark}\label{tthr2}
We have chosen the suggestive notation $\hat{\rho}(0)$ in (\ref{ttheq18}) and (\ref{ttheq19}) because, as follows from \cite[Thm.~10, p.~569]{vindas-estrada2007}, the relation (\ref{ttheq19}) implies that $\hat{\rho}(0)$ is precisely the distributional point value (in the sense of \L ojasiewicz) of the Fourier transform of $\rho$ at the point $t=0$. It should also be noticed that (\ref{ttheq17}) above actually becomes equivalent to 
\begin{equation}
\label{ttheq21} \frac{\mathcal{L}\{\rho,s\}}{s-it}  \quad \mbox{has pseudomeasure boundary behavior at } it,
\end{equation}
as follows from Theorem \ref{tbth1} because $\int_{0}^{x}e^{-itu}\rho(u)\mathrm{d}u$ must boundedly oscillating under the hypotheses of Corollary \ref{improvedK-ABth}; in fact, $\int_{0}^{x}\rho(u)\mathrm{d}u$ is bounded (see below) and the claim follows from integration by parts.
\end{remark}
\begin{proof} Boundedness from below of $\rho$ gives in particular that $\tau(x)=\int_{0}^{x}\rho(u)\mathrm{d}u$ is slowly decreasing. We obtain from Theorem \ref{tbth1} that $\tau$ is a bounded function. This allows us to apply integration by parts in (\ref{ttheq17}) to conclude that 
$$\left|\int_{0}^{x} \tau(u)e^{-it u}\mathrm{d}u\right|\leq \frac{1}{t}\left(\|\tau\|_{L^{\infty}(\mathbb{R})} +M_t\right).$$
The rest follows from Theorem \ref{tth1}.
\end{proof}

In the case of Laplace transforms of bounded functions, we also provide a finite version of Corollary \ref{improvedK-ABth}. We only state the result for $s=0$, but of course other boundary points $s=it_0$ can be treated by replacing $\rho(x)$ by $e^{-it_0 x}\rho(x)$.
 
\begin{theorem}
\label{finiteform}
 Let $\rho\in L^{1}_{loc}(\mathbb{R})$ vanish on $(-\infty,0)$ and be such that
 \begin{equation}
\label{eqbdf}
M:=\limsup_{x\to\infty} |\rho(x)|<\infty.
\end{equation}
Suppose that there are $\lambda>0$, a closed null set $0\not\in E\subset\mathbb{R}$ such that $(\ref{ttheq17})$ holds for each  $t\in E\cap(-\lambda,\lambda)$,
and a constant $\hat{\rho}(0)$ such that
\begin{equation}
\label{ttheq22} \frac{\mathcal{L}\{\rho; s\}-\hat{\rho}(0)}{s}
\end{equation}
has local pseudofunction boundary behavior on $i((-\lambda,\lambda)\setminus E)$. Then, there is an absolute constant $0<\mathfrak{C}\leq 2 $ such that  
\begin{equation}
\label{ttheq23} \limsup_{x\to\infty}\left|\int_{0}^{x}\rho(u)\mathrm{d}u-\hat{\rho}(0)\right|\leq\frac{\mathfrak{C}M}{\lambda} .
\end{equation}
\end{theorem}
\begin{proof} Replacing $\rho$ by $\rho(x/\lambda)/(M+\varepsilon)$, if necessary, and then taking $\varepsilon\to 0^{+}$ in the argument below, we may suppose that $M=\lambda = 1$ and $|\rho(x)|\leq 1$ for all sufficiently large $x$. Define\footnote{Our proof in fact shows that we may state this result for $\tau$ being merely R-slowly oscillating, in this case one gets $\limsup_{x\to\infty}\left|\tau(x)-\hat{\rho}(0)\right|\leq\frac{\mathfrak{C}\Psi'(0^{+})}{\lambda}$, where $\Psi$ is given by (\ref{eqsg}).} 
  $\tau(x)=\hat{\rho}(0)+\int_{0}^{x}\rho(u)\mathrm{d}u$ for $x>0$.
Theorem \ref{cpfth1} gives us the right to assume that $E=\emptyset$ (cf. Remark \ref{tthr2}). Our hypothesis on the boundary behavior of the Laplace transform of $\tau$ is then that $\hat{\tau} \in PF_{loc}(-1,1)$, or equivalently, that (\ref{ttheq9}) holds for each $\phi\in\mathcal{F}(\mathcal{D}(-1,1))$. 
Note\footnote{Actually, for this two-sided Tauberian condition, it is much easier to show that $\tau$ is bounded than in Theorem \ref{tbth1}.} that Theorem \ref{tbth1} yields $\tau\in L^{\infty}(\mathbb{R})$. A standard density argument shows that (\ref{ttheq9}) holds for each $\phi \in L^{1}(\mathbb{R})$ satisfying  $\operatorname*{supp} \hat{\phi}\subseteq [-1,1]$. We choose the Fej\'{e}r kernel
$$
 \phi(x)=\left( \frac{\sin (x/2)}{x/2} \right)^{2}.
$$
Set $K= \limsup_{h \to \infty} \left|\tau(h)\right|$ and assume $K \geq 2$. Since $\phi$ is non-negative, even and satisfies (as seen by numerical evaluation of the integrals) $2.690 \approx\int^{4}_{0} \phi(x)\mathrm{d}x > \int^{\infty}_{4} \phi(x)\mathrm{d}x \approx 0.452$, we have

\begin{align*}
 \limsup_{h\to \infty} \left|\int^{\infty}_{-\infty} \tau(x+h) \phi(x) \mathrm{d}x\right| & 
\geq 2\left(\int^{2K}_{0}(K-x)\phi(x) \mathrm{d}x - \int^{\infty}_{2K}K \phi(x) \mathrm{d}x\right)\\
& \geq 2\left(\int^{4}_{0}(2-x)\phi(x) \mathrm{d}x - \int^{\infty}_{4}2 \phi(x) \mathrm{d}x\right)\\
& \ \ \ + 2(K-2) \left(\int^{4}_{0} \phi(x) \mathrm{d}x - \int^{\infty}_{4} \phi(x)\mathrm{d}x\right)\\
& \geq 2\left(\int^{4}_{0}(2-x)\phi(x) \mathrm{d}x - \int^{\infty}_{4}2 \phi(x) \mathrm{d}x\right)\\
& \approx 2(1.170 - 0.905) \gg 0,
\end{align*}
contradicting (\ref{ttheq9}). This means that $K < 2$ and thus also ${\mathfrak{C}} \leq 2$. 

\end{proof}

\begin{remark}\label{tthr3} The upper bound 2 given in Theorem \ref{finiteform} for $\mathfrak{C}$ is far from being optimal. The proof method from Theorem \ref{finiteform} can be used to give even better values for $\mathfrak{C}$. Ingham's method from \cite{ingham1935} basically gives $0<\mathfrak{C}\leq 6$ for $L^{1}_{loc}$-boundary behavior. The value $\mathfrak{C} = 2$ was already obtained via Newman's method \cite{a-b,korevaar1982,korevaar2005FR,zagier1997} under the stronger hypothesis of analytic continuation of (\ref{ttheq22}) on $i(-\lambda,\lambda)$; Ransford has also given a related result for power series \cite{ransford1988}. We have not pursued any optimality here, but we mention that it is possible to determine the sharp value of the Tauberian constant $\mathfrak{C}$. The analysis of this problem is however quite involved, as it requires an elaborate study of a certain extremal function, and we postpone it for future  investigations.
\end{remark}

Instead of local pseudofunction boundary behavior of (\ref{ttheq22}) on an open interval, Korevaar works in \cite{korevaar2005FR} with the assumptions
\begin{equation}
\label{kceq1} \hat{\rho}(0):=\lim_{\sigma\to0^{+}} \mathcal{L}\{\rho;\sigma\} \quad \mbox{exists}
\end{equation}
and 
\begin{equation}
\label{kceq2} \frac{\mathcal{L}\{\rho;\sigma+it\}-\mathcal{L}\{\rho;\sigma\}}{it}\end{equation}
converges to a local pseudofunction as $\sigma\to0^{+}$. We can apply exactly the same method employed in the proof of Theorem \ref{finiteform} to extend Korevaar's main result from \cite{korevaar2005FR}:

\begin{corollary}
\label{Korevaar's theorem}
If one replaces  the assumption that $(\ref{ttheq22})$ has local pseudofunction behavior on $i((-\lambda,\lambda)\setminus E)$ in Theorem $\ref{finiteform}$ by $(\ref{kceq1})$ and $(\ref{kceq2})$ converges to a local pseudofunction as $\sigma\to0^{+}$ in $\mathcal{D}'((-\lambda,\lambda)\setminus E)$, then the inequality $(\ref{ttheq23})$ remains valid.  \end{corollary}
\begin{proof}As usual, we may assume that $\rho\in L^{\infty}(\mathbb{R})$. Set $F(s)=\mathcal{L}\{\rho;s\}$. Note that $F$ also has local pseudofunction boundary behavior on $i((-\lambda,\lambda)\setminus E)$ except perhaps at $0$. By Theorem \ref{cpfth1}, we obtain that the local pseudofunction boundary behavior of $F$ actually holds in the larger set $i((-\lambda,\lambda)\setminus \{0\})$; consequently,  (\ref{kceq2}) converges to a local pseudofunction as $\sigma\to0^{+}$ in $\mathcal{D}'((-\lambda,\lambda))$;  let $g$ be its local pseudofunction limit. Fix $\psi\in\mathcal{F}(\mathcal{D}(-\lambda,\lambda))$. Let $\tau_{\sigma}(x) = \int^{x}_{0} \rho(u)e^{-\sigma u} \mathrm{d}u - F(\sigma)$ and $\tau(x) = \int^{x}_{0} \rho(u) \mathrm{d}u$. If $h$ is fixed, 
it follows that
$$
 \left|\int_{0}^{\infty}\tau_{\sigma}(x)\psi(x-h)\mathrm{d}x \right| = \frac{1}{2\pi} \left|\left\langle \frac{F(\sigma+it)-F(\sigma)}{it},e^{iht}\hat{\psi}(-t)\right\rangle\right|.
$$
We can now take $\sigma\to0^{+}$ and apply Lebesgue's dominated convergence theorem to obtain
$$
 \left|\int^{\infty}_{0} (\tau(x) -\hat{\rho}(0))\psi(x-h) \mathrm{d}x \right| =  \frac{1}{2\pi} \left|\left\langle g(t),e^{iht}\hat{\psi}(-t)\right\rangle\right|.
$$
The rest of the proof is exactly the same as that of Theorem \ref{finiteform}.
\end{proof}

We now treat the Tauberian condition (\ref{teq2ingham}). We remark that the next corollary improves Ingham's Tauberian constants from \cite[Thm.~I, p.~464]{ingham1935} as well as it weakens the boundary hypotheses on the Laplace transform. 

\begin{corollary}\label{Ingham theorem}Let $\tau$ be of local bounded variation, vanish on $(-\infty,0)$, have convergent Laplace transform $(\ref{ttheq1})$, and satisfy
\begin{equation*}
\limsup_{x\to\infty}e^{-\theta x}\left|\int_{0^{-}}^{x}e^{\theta u}\mathrm{d}\tau(u)\right|=:\Theta<\infty,
\end{equation*}
where $\theta>0$. 
 Let $G(s)=\int_{0}^{\infty}g(x)e^{-sx}$, where $g$ is a bounded function. Suppose that there are $\lambda>0$ and a closed null set $0\notin E\subset(-\lambda,\lambda)$ such that the analytic function (\ref{ttheq3}) has local pseudofunction boundary behavior on $i((-\lambda,\lambda)\setminus E)$
and for each $it\in i(E\cap(-\lambda,\lambda))$ $(\ref{ttheqbc1})$ is satisfied.
Then,
$$
\limsup_{x\to\infty}\left|\tau(x)- b-\int_{0}^{x}g(y)\mathrm{d}y\right| \leq \left(1+\frac{\theta\mathfrak{C}}{\lambda}\right)\left(\Theta+\theta^{-1}\limsup_{x\to\infty}|g(x)|\right),
$$
where $0<\mathfrak{C}\leq 2$.
\end{corollary}
\begin{proof} Note that $\tau_{1}(x)= \tau(x)-b-\int_{0}^{x}g(y)$ satisfies 
\begin{equation*}
\limsup_{x\to\infty}e^{-\theta x}\left|\int_{0^{-}}^{x}e^{\theta u}\mathrm{d}\tau_{1}(u)\right|\leq \Theta+\theta^{-1}\limsup_{x\to\infty}|g(x)|,
\end{equation*}
so that we may assume $b=0$ and $g=0$.
We retain the notation exactly as in the proof of Corollary \ref{tbc1}. Under our assumption, the absolute value of $T_{\theta}$ has superior limit $\Theta$. We employ (\ref{tbeq13}), so
$s^{-1} \mathcal{L}\{T_{\theta};s\}=(s+\theta)^{-1} \mathcal{L}\{\tau;s\}$ . Theorem \ref{finiteform} applied to $T_{\theta}$ gives $\lim_{x\to\infty}|T^{(-1)}_{\theta}(x)|\leq \Theta \mathfrak{C}/\lambda $ and the result hence follows from (\ref{tbeq12}). 
\end{proof}

\section{Power series}
\label{section power series}
This last section is devoted to power series. We apply our ideas from the previous sections to improve results in the neighborhood of the Katznelson-Tzafriri theorem \cite{katznelson}. 

Let us start with some preliminaries. We identify functions and distributions on the unit circle of the complex plane with ($2\pi$-)periodic functions and distributions on the real line. Thus, every periodic distribution can be expanded as a Fourier series \cite{vladimirov} 
\begin{equation}
\label{pseq1}
f(\theta)= \sum_{n=-\infty}^{\infty} c_n e^{in\theta},
\end{equation}
where the Fourier coefficients satisfy the growth estimates $c_n=O(|n|^{k})$ for some $k$. Conversely, if a (two-sided) sequence $\{c_n\}_{n\in\mathbb{Z}}$ has this growth property, then (\ref{pseq1}) defines a (tempered) distribution. Let $\mathbb{D}$ be the unit disc. If $F(z)$ is analytic in $\mathbb{D}$ then distributional boundary values on an open arc of the unit circle $\partial\mathbb{D}$ are defined via the distributional limit $\lim_{r\to1^{-}}F(re^{i\theta})$. We call a periodic distribution $f$ a pseudofunction (pseudomeasure) on $\partial\mathbb{D}$ if $f\in PF_{loc}(\mathbb{R})$ ($f\in PM_{loc}(\mathbb{R})$). It is not hard to verify that the latter holds if and only if its Fourier coefficients tend to 0 (are bounded). We include a proof of this simple fact for the sake of completeness.
\begin{proposition}\label{cpfperiodic} A $2\pi$-periodic distribution with Fourier series $(\ref{pseq1})$ is a pseudofunction (pseudomeasure) on $\partial\mathbb{D}$ if and only if  $c_{n}=o(1)$ ($c_n=O(1)$).
\end{proposition}
\begin{proof} We only give the proof in the pseudofunction case, the pseudomeasure one can be treated similarly. We have that $f\in PF_{loc}(\mathbb{R})$ if and only if
\begin{equation}
\label{pseq2}
\langle f(\theta),e^{-ih\theta}\varphi(-\theta)\rangle=\sum_{n=-\infty}^{\infty}c_{n}\hat{\varphi}(h-n)= o(1), \quad |h|\to\infty,
\end{equation}
for each $\varphi\in\mathcal{D}(\mathbb{R})$. The latter certainly holds if $c_{n}=o(1)$. For the direct implication, we select in  (\ref{pseq2}) a test function $\varphi$ with $\hat{\varphi}(j)=\delta_{0,j}$. (For instance, 
$$\hat{\varphi}(\xi)=\hat{\psi}(\xi) \frac{\sin (\pi \xi )}{\pi \xi} $$  
with an arbitrary test function $\psi\in\mathcal{D}(\mathbb{R})$ such that $\hat{\psi}(0)=1$ satisfies these requirements). Setting $h=N\in\mathbb{Z}$, we obtain $c_N=\langle f(\theta),e^{-iN\theta}\varphi(-\theta)\rangle=o(1)$.
\end{proof}

We are ready to discuss Tauberian theorems. The classical Fatou-Riesz theorem for power series states that if $F(z)=\sum_{n=0}^{\infty}c_nz^{n}$ is convergent on $|z|<1$, has analytic continuation to a neighborhood of $z=1$, and the coefficients satisfy the Tauberian condition $c_n\to 0$, then $\sum_{n=0}^{\infty}c_n$ converges to $F(1)$. The boundary behavior has been weakened \cite[Prop.~14.3, p.~157]{korevaarbook} to local pseudofunction boundary behavior of $(F(z)-F(1))/(z-1)$ near $z=1$ (for some suitable constant $F(1)$). As an application of Theorem \ref{tth2}, we can further relax the Tauberian condition on the coefficients. We can also refine a boundedness theorem of Korevaar \cite[Prop.~III.14.3, p.~157] {korevaarbook} by replacing boundedness of the Taylor coefficients by a one-sided bound. 

\begin{theorem}
\label{psth1} Let $F(z)=\sum_{n=0}^{\infty}c_nz^{n}$ be analytic on the unit disc $\mathbb{D}$. 
\begin{itemize}
\item [(i)] Suppose the sequence $\{c_n\}_{n=0}^{\infty}$ is bounded from below.
If 
\begin{equation*}
\frac{F(z)}{z-1} \quad \mbox{has pseudomeasure boundary behavior at } z=1,
\end{equation*}
then $\sum_{n=0}^{N}c_n=O(1)$. 
\item [(ii)] Suppose that $\liminf_{n\to\infty}c_n\geq 0.
$  
If there is a constant $F(1)$ such that
\begin{equation*}
\frac{F(z)-F(1)}{z-1} \quad \mbox{has pseudofunction boundary behavior at } z=1,
\end{equation*}
then $\sum_{n=0}^{\infty}c_n$ converges to $F(1)$. 

\end{itemize}

\end{theorem}
\begin{remark}\label{psrk1} The converses of Theorem \ref{psth1}(i) and Theorem \ref{psth1}(ii) trivially hold: If $\sum_{n=0}^{\infty}c_n=F(1)$ ($\sum_{n=0}^{N}c_n=O(1)$), then the function $(F(z)-F(1))/(z-1)$ (the function $F(z)/(z-1)$) has global pseudofunction boundary behavior (global pseudomeasure boundary behavior) on $\partial\mathbb{D}$.
\end{remark}
\begin{proof} Set $\tau(x)=\sum_{n\leq x}c_n$. Under the hypotheses of part (i), this function is boundedly decreasing and it has Laplace transform
$$
\mathcal{L}\{\tau;s\}=\frac{1-e^{-s}}{s}\cdot \frac{F(e^{-s})}{1-e^{-s}}, \quad \Re e\:s>0,
$$
with pseudomeasure boundary behavior at $s=0$ because analytic functions are multipliers for local pseudomeasures. That the partial sums are bounded follows from Theorem \ref{tbth1}.
In part (ii), $\tau$ is clearly very slowly decreasing and  
\begin{equation*}
\mathcal{L}\{\tau;s\}-\frac{F(1)}{s}=\frac{1-e^{-s}}{s}\cdot \frac{F(e^{-s})-F(1)}{1-e^{-s}}, \quad \Re e\:s>0,
\end{equation*}
has pseudofunction boundary behavior at $s=0$. Theorem \ref{tth2} then yields $\sum_{n=0}^{\infty}c_n=\lim_{x\to\infty}\tau(x)=F(1)$. \end{proof}
The Katznelson-Tzafriri theorem \cite[Thm.~2$'$, p.~317]{katznelson} allows one to conclude that an analytic function $F(z)$ in $\partial \mathbb{D}$ has pseudofunction boundary behavior on $\partial\mathbb{D}$ from pseudofunction boundary behavior except at $z=1$ plus the additional assumption that the partial sums of its Taylor coefficients form a bounded sequence. Extensions of this theorem were obtained in \cite{a-o-r} and \cite[Sect. 13 and 14, Chap.~III]{korevaarbook}. The ensuing theorem contains all of those results. Indeed, Theorem \ref{psth2} removes earlier unnecessary uniformity assumptions on possible boundary singularity sets in a theorem by Allan, O'Farell, and Ransford \cite{a-o-r} (cf. also \cite[Thm.~ III.14.5, p.~159]{korevaarbook}), and furthermore relaxes the $H^1$-boundary behavior to pseudofunction boundary behavior.  

\begin{theorem}\label{psth2}  Let $F(z)=\sum_{n=0}^{\infty}c_nz^{n}$ be analytic in the unit disc $\mathbb{D}$.
 Suppose that there is a closed subset $E\subset \partial\mathbb{D}$ of null (linear) measure such that $F$ has local pseudofunction boundary behavior on $\partial\mathbb{D}\setminus E$, whereas for each $e^{i\theta}\in E$ the bound 
\begin{equation}
\label{pseq3}
\sum^{N}_{n=0}c_{n}e^{in\theta }=O_{\theta}(1)
\end{equation}
holds. Then, $F$ has pseudofunction boundary behavior on the whole $\partial\mathbb{D}$, that is, $c_n=o(1)$. In particular, $\sum^{\infty}_{n=0}c_{n}e^{in \theta_0 }$ converges at every point where there is a constant $F(e^{i\theta_{0}})$ such that
\begin{equation*}
\frac{F(z)-F(e^{i\theta_{0}})}{z-e^{i\theta_0 }}
\end{equation*}
has pseudofunction boundary behavior at $z=e^{i\theta_0}\in\partial\mathbb{D}$, and moreover
$$
\sum^{\infty}_{n=0}c_{n}e^{in \theta_0 }=F(e^{i\theta_0}).
$$
\end{theorem}
\begin{proof}  Since (\ref{pseq3}) implies that 
\begin{equation}
\label{pseq4}
\frac{F(z)}{z-e^{i\theta }} \quad \mbox{has pseudomeasure boundary behavior at }e^{i\theta}\in\partial\mathbb{D},
\end{equation}
the first assertion follows by combining Theorem \ref{cpfth1} and Proposition \ref{cpfperiodic}. For the last claim, since we now know that $c_n=o(1)$, we can suppose that $\theta_0=0$ and, by splitting into real and imaginary parts, that the $c_n$ are real-valued. The convergence of $\sum_{n=0}^{\infty}c_n $ is thus a direct consequence of Theorem \ref{psth1}(ii).
\end{proof}
We end this article with a comment about (\ref{pseq3}).
\begin{remark}
\label{psrk2} Naturally, Theorem \ref{psth2} also holds if we replace (\ref{pseq3}) by the weaker assumption (\ref{pseq4}) for each $e^{i\theta}\in E$. On the other hand, if the coefficients $\{c_n\}_{n=0}^{\infty}$ are known to be bounded, then condition (\ref{pseq4}) becomes equivalent to  (\ref{pseq3}), as follows as in the proof of Theorem \ref{psth1}(i) by applying Theorem \ref{tbth1} to the boundedly oscillating function $\tau(x)=\sum_{n\leq x}c_ne^{in\theta}$. \end{remark}

\end{document}